\documentclass[11pt]{amsart}
\usepackage{amsmath}
\usepackage{amsthm}
\usepackage{tikz}
\usepackage{hyperref}
\usepackage[left=2.4cm,right=2.4cm,top=2.8cm,bottom=2.8cm]{geometry}
\usepackage{eucal}
\usepackage{mathrsfs}
\usepackage{amssymb}
\usepackage{amscd}
\usepackage{latexsym}
\usepackage{epsfig}
\usepackage{graphicx}
\usepackage{caption}
\usepackage{rotating}
\usepackage{xypic}
\usepackage{stmaryrd}

\usepackage{epsfig}

\usepackage[labelsep=period]{caption}
\usepackage{url}
\usepackage{cite}
\usepackage{array,}

\newtheorem{theorem}{Theorem}
\newtheorem{proposition}[theorem]{Proposition}
\newtheorem{corollary}[theorem]{Corollary}
\newtheorem{lemma}[theorem]{Lemma}

\theoremstyle{definition}
\newtheorem{definition}[theorem]{Definition}
\newtheorem{example}[theorem]{Example}
\newtheorem{remark}[theorem]{Remark}
\newtheorem{observation}[theorem]{Observation}

\newcommand{\laa}{\mathfrak{a}}
\newcommand{\lab}{\mathfrak{b}}
\newcommand{\lac}{\mathfrak{c}}

\newcommand{\lae}{\mathfrak{e}}
\newcommand{\laf}{\mathfrak{f}}
\newcommand{\lag}{\mathfrak{g}}
\newcommand{\lah}{\mathfrak{h}}
\newcommand{\lai}{\mathfrak{i}}
\newcommand{\laj}{\mathfrak{j}}

\newcommand{\lal}{\mathfrak{l}}

\newcommand{\lap}{\mathfrak{p}}
\newcommand{\laq}{\mathfrak{q}}

\newcommand{\las}{\mathfrak{s}}
\newcommand{\lat}{\mathfrak{t}}

\newcommand{\laz}{\mathfrak{z}}

\newcommand{\lagl}{\mathfrak{gl}}
\newcommand{\lasl}{\mathfrak{sl}}
\newcommand{\laso}{\mathfrak{so}}

\newcommand{\lasp}{\mathfrak{sp}}

\usepackage{microtype}

\DeclareMathOperator{\ess}{ess}
\DeclareMathOperator{\diag}{diag}
\DeclareMathOperator{\bbC}{\mathbb{C}}
\DeclareMathOperator{\bbZ}{\mathbb{Z}}

\DeclareMathOperator{\GL}{GL}
\DeclareMathOperator{\SL}{SL}

\DeclareMathOperator{\adj}{ad}

\DeclareMathOperator{\Hom}{Hom}

\DeclareMathOperator{\reg}{\mathrm{reg}}

\newcommand{\ssslash}{{/\!\!/\!\!/}}
\newcommand{\dslash}{{/\!\!/}}

\title[Hyperk\"ahler slices and spherical geometry]{An application of spherical geometry to \\ hyperk\"ahler slices}

\author[Peter Crooks]{Peter Crooks}
\address[Crooks]{Department of Mathematics, Northeastern University, 360 Huntington Ave., Boston, MA 02115, USA}
\email{~~~peter.d.crooks@gmail.com}

\author[Maarten van Pruijssen]{Maarten van Pruijssen}
\address[van Pruijssen]{Institut f\"ur Mathematik, University of Paderborn, Warburger Str. 100, 33098 Paderborn, Germany}
\email{~~~vanpruijssen@math.upb.de}


\begin{document}

\maketitle

\begin{abstract}
This work is concerned with Bielawski's hyperk\"ahler slices in the cotangent bundles of homogeneous affine varieties. One can associate such a slice to the data of a complex semisimple Lie group $G$, a reductive subgroup $H\subseteq G$, and a Slodowy slice $S\subseteq\mathfrak{g}:=\mathrm{Lie}(G)$, defining it to be the hyperk\"ahler quotient of $T^*(G/H)\times (G\times S)$ by a maximal compact subgroup of $G$. This hyperk\"ahler slice is empty in some of the most elementary cases (e.g. when $S$ is regular and $(G,H)=(\SL_{n+1},\GL_{n})$, $n\geq 3$), prompting us to seek necessary and sufficient conditions for non-emptiness.

We give a spherical-geometric characterization of the non-empty hyperk\"ahler slices that arise when $S=S_{\text{reg}}$ is a regular Slodowy slice, proving that non-emptiness is equivalent to the so-called $\mathfrak{a}$\textit{-regularity} of $(G,H)$. This $\mathfrak{a}$-regularity condition is formulated in several equivalent ways, one being a concrete condition on the rank and complexity of $G/H$. We also provide a classification of the $\laa$-regular pairs $(G,H)$ in which $H$ is a reductive spherical subgroup. Our arguments make essential use of Knop's results on moment map images and Losev's algorithm for computing Cartan spaces.
\end{abstract}

\tableofcontents

\section{Introduction}
\subsection{Context}
A smooth manifold is called \textit{hyperk\"ahler} if it comes equipped with three K\"ahler structures that determine the same Riemannian metric, and whose underlying complex structures satisfy certain quaternionic identities. Such manifolds are known to be holomorphic symplectic and Calabi--Yau, and they are ubiquitous in modern algebraic and symplectic geometry. Prominent examples include the cotangent bundles \cite{Kronheimer} and (co-)adjoint orbits \cite{Kovalev,KronheimerNilpotent,KronheimerSemisimple,Biquard} of complex semisimple Lie groups, moduli spaces of Higgs bundles over compact Riemann surfaces \cite{HitchinSelf}, and Nakajima quiver varieties \cite{Nakajima1,Nakajima2}. Many examples arise via the \textit{hyperk\"ahler quotient} construction\cite{Hitchin}, an analogue of symplectic reduction for a hyperk\"ahler manifold endowed with a structure-preserving Lie group action and a hyperk\"ahler moment map. However, one always has the preliminary problem of determining whether the given hyperk\"ahler quotient is non-empty.    

While the above-described \textit{emptiness problem} is likely intractable in the generality described above, one might hope to solve it for particular classes of hyperk\"ahler quotients. It is in this context that one might consider Bielawski's \textit{hyperk\"ahler slices}\cite{Bielawski,BielawskiComplex}, which require fixing a compact, connected, semisimple Lie group $K$ with complexification $G:=K_{\mathbb{C}}$. Each $\mathfrak{sl}_2$-triple $\tau=(\xi,h,\eta)$ in $\mathfrak{g}:=\mathrm{Lie}(G)$ determines a Slodowy slice $S_{\tau}:=\xi+\mathrm{ker}(\adj_{\eta})\subseteq\mathfrak{g}$, and hence also an affine variety $G\times S_{\tau}$. This variety is a hyperk\"ahler manifold carrying a tri-Hamiltonian action of $K$, and its symplectic geometry is reasonably well-studied (see \cite{Bielawski, CrooksRayan, AbeCrooks, CrooksBulletin}). Now suppose that $K$ acts in a tri-Hamiltonian fashion on a hyperk\"ahler manifold $M$, and that this action extends to a holomorphic, Hamiltonian $G$-action with respect to the holomorphic symplectic structure on $M$. The \textit{hyperk\"ahler slice} for $M$ and $\tau$ is then defined to be $(M\times (G\times S_{\tau}))\ssslash K$, the hyperk\"ahler quotient of $M\times (G\times S_{\tau})$ by $K$. Several well-known hyperk\"ahler manifolds are realizable as hyperk\"ahler slices, as discussed in the introduction of \cite{BielawskiComplex}.  

In light of the preceding discussion, one might consider the following special case of the emptiness problem: classify those pairs $(M,\tau)$ for which the hyperk\"ahler slice $(M\times (G\times S_{\tau}))\ssslash K$ is non-empty. An initial objection is that no particular assumptions have been made about $M$ and $\tau$, so that this problem likely remains too general to be tractable. We thus note that the best studied Slodowy slices are those associated to \textit{regular} $\mathfrak{sl}_2$-triples $\tau$ (see \cite{Kostant}), i.e. those $\tau=(\xi,h,\eta)$ for which $\xi$ is a regular element of $\mathfrak{g}$. At the same time, some of the best understood hyperk\"ahler manifolds take the form of $T^*(G/H)$ for $H\subseteq G$ a closed, reductive subgroup (see \cite{Dancer}). We therefore study the emptiness problem for hyperk\"ahler slices when $\tau$ is regular and $M=T^*(G/H)$. 

Having decided to study hyperk\"ahler slices in $T^*(G/H)$, we are naturally led to examine the Hamiltonian geometry of $T^*(G/H)$. The works of Knop \cite{KnopWeyl, KnopAsymptotic} encode this Hamiltonian geometry in the spherical geometry of $G/H$, by which we mean the $B$-orbit structure of $G/H$ for a Borel subgroup $B\subseteq G$. Fix such a subgroup $B\subseteq G$ and a maximal torus $T\subseteq B$ having Lie algebra $\lat\subseteq\lag$. Knop uses the \textit{Cartan space} $\mathfrak{a}_{G/H}^*\subseteq\mathfrak{t}^*$ to describe the (closure of the) moment map image of $T^*(G/H)$. This is complemented by Losev's work \cite{MR2362821}, which gives an algorithm for calculating the Cartan space of any given affine homogeneous $G$-variety \cite{MR2362821}. It is thus reasonable to imagine that spherical-geometric ideas are relevant to our specific emptiness problem.

\subsection{Description of results}\label{Subsection: Description of results}
Let all notation be as set in the previous subsection, and write $S_{\text{reg}}$ for the Slodowy slice determined by a regular $\mathfrak{sl}_2$-triple $\tau$ in $\mathfrak{g}$. Use the Killing form to identify $\mathfrak{g}^*$ with $\mathfrak{g}$, and let $\mu:T^*(G/H)\rightarrow\mathfrak{g}$ be the moment map of the Hamiltonian $G$-action on $T^*(G/H)$. We note the existence of a non-negative, $K$-invariant potential function for the first K\"ahler triple on $T^*(G/H)$ (Proposition \ref{Proposition: Kahler potential}), which by Bielawski's results \cite{Bielawski} implies that $(T^*(G/H)\times (G\times S_{\text{reg}}))\ssslash K$ and $\mu^{-1}(S_{\text{reg}})$ are canonically isomorphic as holomorphic symplectic manifolds. This isomorphism is subsequently used to prove that $(T^*(G/H)\times (G\times S_{\text{reg}}))\ssslash K\neq\emptyset$ if and only if $\mathfrak{h}^{\perp}$ contains a regular element of $\mathfrak{g}$ (Proposition \ref{Proposition: Non-empty slice}), where $\mathfrak{h}^{\perp}\subseteq\mathfrak{g}$ denotes the annihilator of $\mathfrak{h}:=\mathrm{Lie}(H)$ under the Killing form. The emptiness problem for $(T^*(G/H)\times (G\times S_{\text{reg}}))\ssslash K$ thus reduces to classifying the pairs $(G,H)$ for which $\mathfrak{h}^{\perp}$ contains a regular element. This is the stage at which spherical geometry becomes relevant, as we explain below. 

Inside of $G$, fix a maximal torus $T$ and a Borel subgroup $B$ satisfying $T\subseteq B$. These choices allow us to form the \textit{Cartan space} of $G/H$, denoted $\mathfrak{a}_{G/H}\subseteq\mathfrak{t}:=\mathrm{Lie}(T)$. We refer to the pair $(G,H)$ as being $\mathfrak{a}$\textit{-regular} if $\mathfrak{a}_{G/H}$ contains a regular element of $\mathfrak{g}$, and we use
Knop's description of the moment map image $\mu(T^*(G/H))$ to prove the following equivalences (see Proposition \ref{prop: regular element}, Corollary \ref{cor: gen stab}, and Corollary \ref{cor: num crit}): \begin{align}\begin{split}\label{Equation: Equivalences}(G,H)\text{ is }\mathfrak{a}\text{-regular } & \Longleftrightarrow\text{ }\mathfrak{h}^{\perp}\text{ contains a regular element }\\ & \Longleftrightarrow\text{ }Z_G(\mathfrak{a}_{G/H})=T\\ & \Longleftrightarrow \text{the identity component of }H_{*}\text{ is abelian}\\& \Longleftrightarrow c_{G}(G/H)+\mathrm{rk}_{G}(G/H)+\dim H=\dim B,\end{split}\end{align} where $Z_G(\mathfrak{a}_{G/H})$ is the subgroup consisting of all elements in $G$ that fix $\mathfrak{a}_{G/H}$ pointwise, $H_{*}$ is the generic stabilizer for the $H$-representation $\mathfrak{h}^{\perp}$ (see \ref{subsection: The cotangent bundle of a homogeneous space}), $c_G(G/H)$ is the complexity of $G/H$, and $\mathrm{rk}_G(G/H)$ is the rank of $G/H$. The first equivalence further reduces our emptiness problem to one of classifying the $\mathfrak{a}$-regular pairs $(G,H)$, thereby connecting our work to Losev's results \cite{MR2362821}. We then classify all such pairs $(G,H)$ (i.e. we solve the emptiness problem for $(T^*(G/H)\times (G\times S_{\text{reg}}))\ssslash K)$ in each of the following three cases:

\begin{itemize}
\item $G$ is semisimple and $H$ is a Levi subgroup of $G$ (\ref{sss: Levi});  
\item $G$ is semisimple and $H$ is a symmetric subgroup of $G$ (\ref{sss: sym}); 
\item $G$ is semisimple and $H$ is a reductive, spherical, non-symmetric subgroup of $G$ (\ref{sss: sph}).
\end{itemize}

In each case, we reduce to the study of strictly indecomposable (see \ref{Subsection: The Cartan space of a homogeneous affine space}) pairs $(G,H)$. It is in the last two cases that we obtain the most explicit results, and where we provide tables of all $\mathfrak{a}$-regular pairs $(G,H)$ that are strictly indecomposable.

\subsection{Organization}
Section \ref{Section: Preliminaries} establishes some of our conventions regarding symplectic and hyperk\"ahler geometry. Section \ref{Section: The hyperkahler geometry of T*(G/H)} then uses \cite{Dancer}, \cite{Kronheimer}, and \cite{MayrandTG} to develop the hyperk\"ahler-geometric features of $T^*(G/H)$ needed for the subsequent discussion of hyperk\"ahler slices. This leads to Section \ref{Section: The hyperkahler slice construction}, which reviews Bielawski's hyperk\"ahler slice construction and reduces the non-emptiness of $(T^*(G/H)\times (G\times S_{\text{reg}}))\ssslash K$ to the condition that $\mathfrak{h}^{\perp}$ contain a regular element. Section \ref{Section: The spherical geometry of T*(G/H)} then forms the spherical-geometric part of our paper, where we prove the equivalences \eqref{Equation: Equivalences} and subsequently obtain our classification results.       

\subsection*{Acknowledgements}
The central themes of this paper were developed at the Hausdorff Research Institute for Mathematics (HIM), while both authors took part in the HIM--sponsored program \textit{Symplectic geometry and representation theory}. We gratefully acknowledge the HIM for its hospitality and stimulating atmosphere. We also wish to recognize Steven Rayan and Markus R\"oser for enlightening conversations. The first author is supported by the Natural Sciences and Engineering Research Council of Canada [516638--2018].   

\section{Preliminaries}\label{Section: Preliminaries}

\subsection{Symplectic varieties and quotients}\label{Subsection: Symplectic varieties and quotients} Let $(X,\omega)$ be a symplectic variety, which for us shall always mean that $X$ is a smooth affine algebraic variety over $\bbC$ equipped with an algebraic symplectic form $\omega\in\Omega^2(X)$. Suppose that $X$ is acted upon algebraically by a connected complex reductive algebraic group $G$ having Lie algebra $\mathfrak{g}$. We recall that this action is called \textit{Hamiltonian} if it preserves $\omega$ and admits a moment map, i.e. a $G$-equivariant variety morphism $\mu:X\rightarrow\mathfrak{g}^*$ satisfying the following condition:
$$d(\mu^{z})=\iota_{\tilde{z}}\omega$$ for all $z\in\mathfrak{g}$, where $\mu^z:X\rightarrow\mathbb{C}$ is defined by $\mu^z(x):=(\mu(x))(z)$, $x\in X$, and $\tilde{z}$ is the fundamental vector field on $X$ associated to $z$. If the $G$-action is also free, then 
\begin{equation}\label{Equation: Holomorphic symplectic quotient}
X\dslash G:=\mu^{-1}(0)/G:=\mathrm{Spec}_{\text{max}}(\mathbb{C}[\mu^{-1}(0)]^G)
\end{equation}
is a smooth affine variety whose points are precisely the $G$-orbits in $\mu^{-1}(0)$. The quotient variety $X\dslash G$ then carries a symplectic form $\overline{\omega}$ that is characterized by the condition $\pi^*(\overline{\omega})=j^*(\omega)$, where $\pi:\mu^{-1}(0)\rightarrow X\dslash G$ is the quotient map and $j: \mu^{-1}(0)\rightarrow X$ is the inclusion. The symplectic variety $(X\dslash G,\overline{\omega})$ is called the \textit{symplectic quotient} of $X$ by $G$. 

\subsection{Hyperk\"ahler manifolds}
Recall that a smooth manifold $M$ is called \textit{hyperk\"ahler} if it comes equipped with three (integrable) complex structures $I_1$, $I_2$, and $I_3$, three (real) symplectic forms $\omega_1$, $\omega_2$, and $\omega_3$, and a single Riemannian metric $b$,
subject the following conditions: 
\begin{itemize}
\item $(I_{\ell},\omega_{\ell},b)$ is a K\"ahler triple for each $\ell=1,2,3$, i.e. $\omega_{\ell}(\cdot,\cdot)=b(I_{\ell}(\cdot),\cdot)$;
\item $I_1$, $I_2$, and $I_3$ satisfy the quaternionic identities $I_1I_2=I_3=-I_2I_1$, $I_1I_3=-I_2=-I_3I_1$, $I_2I_3=I_1=-I_3I_2$.  
\end{itemize}
One may construct new examples from existing ones via the \textit{hyperk\"ahler quotient} construction, which we now recall. Let $K$ be a compact connected Lie group acting freely on a hyperk\"ahler manifold $M$, and let $\mathfrak{k}$ be the Lie algebra of $K$. Assume that the $K$-action is \textit{tri-Hamiltonian}, meaning that $K$ preserves each K\"ahler triple $(I_{\ell},\omega_{\ell},b)$ and acts in a Hamiltonian fashion with respect to each symplectic form $\omega_{\ell}$. One thus has a \textit{hyperk\"ahler moment map}, i.e. a map $\mu_{\text{HK}}=(\mu_1,\mu_2,\mu_3):M\rightarrow\mathfrak{k}^*\oplus\mathfrak{k}^*\oplus\mathfrak{k}^*$ with the property that $\mu_{\ell}:M\rightarrow\mathfrak{k}^*$ is a moment map for the $K$-action with respect to $\omega_{\ell}$, $\ell=1,2,3$. The smooth manifold $$M\ssslash K:=\mu_{\text{HK}}^{-1}(0)/K=(\mu_1^{-1}(0)\cap\mu_2^{-1}(0)\cap\mu_3^{-1}(0))/K$$ is then canonically hyperk\"ahler (see \cite[Theorem 3.2]{Hitchin}), and it is called the \textit{hyperk\"ahler quotient} of $M$ by $K$. We shall let $(\overline{I}_{\ell},\overline{\omega}_{\ell},\overline{b})$, $\ell=1,2,3$, denote the three K\"ahler triples that constitute the hyperk\"ahler structure on $M\ssslash K$. It will be advantageous to note that \begin{equation}\label{Equation: Symplectic form relations}\pi^*(\overline{\omega}_{\ell})=j^*(\omega_{\ell}),\quad \ell=1,2,3,\end{equation} where $\pi:\mu^{-1}(0)\rightarrow M\ssslash K$ is the quotient map and $j:\mu^{-1}(0)\rightarrow M$ is the inclusion. 

Let $M$ be a hyperk\"ahler manifold and consider the complex symplectic $2$-form $\omega_{\mathbb{C}}:=\omega_2+i\omega_3$. One can verify that $\omega_{\mathbb{C}}$ is holomorphic with respect to $I_1$, and we will refer to $(M,I_1,\omega_{\mathbb{C}})$ as the \textit{underlying holomorphic symplectic manifold}. This leads to the following definition, which will apply to many situations of interest in our paper.

\begin{definition}\label{Definition: (G,K)-hyperkahler variety}
Let $K$ be a compact connected Lie group with complexification $G:=K_{\mathbb{C}}$. We define a $(G,K)$\textit{-hyperk\"ahler variety} is a to be a hyperk\"ahler manifold $M$ satisfying the following conditions:
\begin{itemize}
\item[(i)] the underlying holomorphic symplectic manifold is a symplectic variety (as defined in \ref{Subsection: Symplectic varieties and quotients}), and this variety is equipped with a Hamiltonian action of $G$;
\item[(ii)] the $G$-action restricts to a tri-Hamiltonian action of $K$ on $M$.  
\end{itemize}
\end{definition}  

Consider the hyperk\"ahler moment map $\mu_{\text{HK}}=(\mu_1,\mu_2,\mu_3):M\rightarrow\mathfrak{k}^*\oplus\mathfrak{k}^*\oplus\mathfrak{k}^*$ on a $(G,K)$-hyperk\"ahler variety $M$. Define the \textit{complex moment map} by $$\mu_{\mathbb{C}}:=\mu_2+i\mu_3:M\rightarrow\mathfrak{k}^*\otimes_{\mathbb{R}}\mathbb{C}=\mathfrak{g}^*,$$ which turns out to be the moment map for the Hamiltonian $G$-action on $M$. Now assume that this $G$-action is free. The inclusion $\mu_{\text{HK}}^{-1}(0)\subseteq\mu_{\mathbb{C}}^{-1}(0)$ then induces a map \begin{equation}\label{Equation: Hyperkahler and holomorphic symplectic quotients} \varphi:M\ssslash K\rightarrow M\dslash G,\end{equation} where we recall that $M\dslash G$ is defined via \eqref{Equation: Holomorphic symplectic quotient}. This map defines a diffeomorphism from $M\ssslash K$ to its image, the open subset $(G\cdot \mu_{\text{HK}}^{-1}(0))/G$ of $\mu_{\mathbb{C}}^{-1}(0)/G=M\dslash G$. Furthermore, $\varphi$ is an embedding of holomorphic symplectic manifolds with respect to the underlying holomorphic symplectic structure on $M\ssslash K$.

\section{The hyperk\"ahler geometry of $T^*(G/H)$}\label{Section: The hyperkahler geometry of T*(G/H)}
It will be convenient to standardize some of the Lie-theoretic notation used in this paper. Let $K$ be a compact connected semisimple Lie group, and fix a closed subgroup $L\subseteq K$. We will also let $G:=K_{\mathbb{C}}$ and $H:=L_{\mathbb{C}}$ denote the complexifications of $K$ and $L$, respectively, noting that $H$ is a closed reductive subgroup of $G$. Let $\mathfrak{k}$, $\mathfrak{l}$, $\mathfrak{g}$, and $\mathfrak{h}$ be the Lie algebras of $K$, $L$, $G$, and $H$, respectively, so that $\mathfrak{g}=\mathfrak{k}\otimes_{\mathbb{R}}\mathbb{C}$ and $\mathfrak{h}=\mathfrak{l}\otimes_{\mathbb{R}}\mathbb{C}$. Each of these Lie algebras comes equipped with the adjoint representation of the corresponding group, e.g. $\mathrm{Ad}:G\rightarrow\GL(\mathfrak{g})$, $g\mapsto\mathrm{Ad}_g$. The symbol ``$\mathrm{Ad}$'' will be used for all of the aforementioned adjoint representations, as context will always clarify any ambiguities that this abuse of notation may cause.   

Let $\langle\cdot,\cdot\rangle:\mathfrak{g}\otimes_{\mathbb{C}}\mathfrak{g}\rightarrow\mathbb{C}$ denote the Killing form on $\mathfrak{g}$, which is $G$-invariant and non-degenerate. It follows that 
\begin{equation}\label{Equation: Killing isomorphism} \mathfrak{g}\rightarrow\mathfrak{g}^*,\quad x\mapsto x^{\vee}:=\langle x,\cdot\rangle,\quad x\in\mathfrak{g}\end{equation} defines an isomorphism between the adjoint and coadjoint representations of $G$. With this in mind, we will sometimes take the moment map for a Hamiltonian $G$-action to be $\mathfrak{g}$-valued. 

\subsection{The cotangent bundle of $G$}
Note that left and right multiplication give the commuting actions
\begin{subequations}
\begin{align}
& g\cdot h:=gh,\quad g,h\in G\label{Equation: Left multiplication}\\
& g\cdot h:=hg^{-1},\quad g,h\in G\label{Equation: Right multiplication}
\end{align}
\end{subequations}
of $G$ on itself, and that these lift to commuting Hamiltonian actions of $G$ on $T^*G$. To be more explicit about this point, we shall use the left trivialization of $T^*G$ and the Killing form to identify $T^*G$ with $G\times\mathfrak{g}$. The lifts of \eqref{Equation: Left multiplication} and \eqref{Equation: Right multiplication} then become
\begin{subequations}
\begin{align}
& g\cdot (h,x)=(gh,x),\quad g\in G\text{ },(h,x)\in G\times\mathfrak{g},\label{Equation: Left lift}\\
& g\cdot (h,x)=(hg^{-1},\mathrm{Ad}_g(x)),\quad g\in G\text{ },(h,x)\in G\times\mathfrak{g}\label{Equation: Right lift},
\end{align}
\end{subequations}
respectively, while the induced symplectic form on $G\times\mathfrak{g}$ is defined on each tangent space $T_{(g,x)}(G\times\mathfrak{g})=T_gG\oplus\mathfrak{g}$ as follows (see \cite[Section 5, Equation (14L)]{Marsden}):
\begin{equation}\label{Equation: Left structure}(\Omega_{L})_{(g,x)}\bigg((d_eL_g(y_1),z_1),(d_eL_g(y_2),z_2)\bigg)=\langle y_1,z_2\rangle-\langle y_2,z_1\rangle+\langle x,[y_1,y_2]\rangle,\quad y_1,y_2,z_1,z_2\in\mathfrak{g},\end{equation}
where $L_g:G\rightarrow G$ denotes left multiplication by $g$ and $d_eL_g:\mathfrak{g}\rightarrow T_gG$ is the differential of $L_g$ at the identity $e\in G$ (see .
One can then verify that
\begin{subequations}
\begin{align}
& \phi_L:G\times\mathfrak{g}\rightarrow\mathfrak{g},\quad (g,x)\mapsto\mathrm{Ad}_g(x),\quad (g,x)\in G\times\mathfrak{g}\label{Equation: Left moment map},\\
& \phi_R:G\times\mathfrak{g}\rightarrow\mathfrak{g},\quad (g,x)\mapsto -x,\quad (g,x)\in G\times\mathfrak{g}\label{Equation: Right moment map}
\end{align}
\end{subequations}
are moment maps for \eqref{Equation: Left lift} and \eqref{Equation: Right lift}, respectively. 

\subsection{Kronheimer's hyperk\"ahler structure on $T^*G$}\label{Subsection: Kronheimer's structure}
Let $\mathbb{H}$ denote the quaternions, to be identified as a vector space with $\mathbb{R}^4$ via the usual basis $\{1,i,j,k\}$. Now consider the real vector space $C^{\infty}([0,1],\mathfrak{k})$ of all smooth maps $[0,1]\rightarrow\mathfrak{k}$. A choice of $K$-invariant inner product $\langle\cdot,\cdot\rangle_{\mathfrak{k}}$ on $\mathfrak{k}$ makes $\mathcal{M}:=C^{\infty}([0,1],\mathfrak{k})\otimes_{\mathbb{R}}\mathbb{H}=C^{\infty}([0,1],\mathfrak{k})^{\oplus 4}$ into a Banach space with an infinite-dimensional hyperk\"ahler manifold structure. This space carries the following hyperk\"ahler structure-preserving action of $\mathcal{G}:=C^{\infty}([0,1],K)$, the \textit{gauge group} of smooth maps $[0,1]\rightarrow K$ with pointwise multiplication as the group operation:
\begin{equation}\label{Equation: Gauge group action}\gamma\cdot(T_0,T_1,T_2,T_3):=(\mathrm{Ad}_{\gamma}(T_0)-\theta_R(\dot{\gamma}),\mathrm{Ad}_{\gamma}(T_1), \mathrm{Ad}_{\gamma}(T_2), \mathrm{Ad}_{\gamma}(T_3)),\quad \gamma\in\mathcal{G},\text{ }(T_0,T_1,T_2,T_3)\in\mathcal{M},
\end{equation}
where $\theta_R\in\Omega^1(K;\mathfrak{k})$ is the right-invariant Mauer--Cartan form on $K$. The subgroup $$\mathcal{G}_0:=\{\gamma\in\mathcal{G}:\gamma(0)=e=\gamma(1)\}\subseteq\mathcal{G}$$ then acts freely on $\mathcal{M}$ with a hyperk\"ahler moment map that can be written in the form $\Phi:\mathcal{M}\rightarrow C^{\infty}([0,1],\mathfrak{k})^{\oplus 3}$. It turns out that $\Phi^{-1}(0)$ consists of the solutions to \textit{N\"ahm's equations} (as defined in \cite[Proposition 1]{Dancer}, for example), and that Kronheimer constructed an explicit diffeomorphism
\begin{equation}\label{Equation: Explicit diffeomorphism}G\times\mathfrak{g}\cong \mathcal{M}\ssslash\mathcal{G}_0=\Phi^{-1}(0)/\mathcal{G}_0
\end{equation}
(cf. \cite[Proposition 1]{Kronheimer}). The smooth manifold $G\times\mathfrak{g}$ thereby inherits a hyperk\"ahler structure $(I_{\ell},\omega_{\ell},b)$, $\ell=1,2,3$. We note that $\omega_2+i\omega_3$ equals the form $\Omega_L$ from \eqref{Equation: Left structure}, while $I_1$ is the usual complex structure on $G\times\mathfrak{g}$ (see \cite[Section 2]{Kronheimer}).

Kronheimer's diffeomorphism \ref{Equation: Explicit diffeomorphism} has some important equivariance properties that we now discuss. Note that $\mathcal{G}_0$ is the kernel of
$$\mathcal{G}\rightarrow K\times K,\quad \gamma\mapsto(\gamma(0),\gamma(1)),\quad\gamma\in\mathcal{G},$$
so that we may identify $\mathcal{G}/\mathcal{G}_0$ and $K\times K$ as Lie groups. The $\mathcal{G}$-action on $\mathcal{M}$ induces a residual action of $\mathcal{G}/\mathcal{G}_0=K\times K$ on $\mathcal{M}\ssslash\mathcal{G}_0$, and this residual action is known to be tri-Hamiltonian (see \cite[Lemma 2]{Dancer}). Under \eqref{Equation: Explicit diffeomorphism}, the action of $K=\{e\}\times K\subseteq K\times K$ on $\mathcal{M}\ssslash\mathcal{G}_0$ corresponds to the $K$-action \eqref{Equation: Left lift} on $G\times\mathfrak{g}$. The diffeomorphism also intertwines the action of $K=K\times\{e\}\subseteq K\times K$ on $M\ssslash\mathcal{G}_0$ with the $K$-action \ref{Equation: Right lift} on $G\times\mathfrak{g}$.

The group $\operatorname{SO}_3(\mathbb{R})$ also has a natural manifestation in our setup. Given a point $(T_0,T_1,T_2,T_3)\in\mathcal{M}=C^{\infty}([0,1],\mathfrak{k})^{\oplus 4}$ and a matrix $A=(a_{pq})\in \operatorname{SO}_3(\mathbb{R})$, let us set $$T_{p}':=\sum_{q=1}^3a_{pq}T_q,\quad p=1,2,3\quad\text{ and }\quad A\cdot (T_0,T_1,T_2,T_3):=(T_0,T_1',T_2',T_3').$$ This action of $\operatorname{SO}_3(\mathbb{R})$ on $\mathcal{M}$ descends to an isometric action on the hyperk\"ahler quotient $\mathcal{M}\ssslash\mathcal{G}_0$. One can use \eqref{Equation: Explicit diffeomorphism} to interpret this as an isometric action of $\operatorname{SO}_3(\mathbb{R})$ on the hyperk\"ahler manifold $G\times\mathfrak{g}$, and it is not difficult to check that this action commutes with the $K$-actions \eqref{Equation: Left lift} and \eqref{Equation: Right lift}. It is important to note that $\operatorname{SO}_3(\mathbb{R})$ does not preserve all of the hyperk\"ahler structure on $G\times\mathfrak{g}$, in contrast to the $K$-actions. However, one can find a circle subgroup of $\operatorname{SO}_3(\mathbb{R})$ that preserves the K\"ahler triple $(I_3,\omega_3,b)$ on $G\times\mathfrak{g}$. A more explicit statement is that one can find an element $\theta\in\mathfrak{so}_3(\mathbb{R})$ whose fundamental vector field $\tilde{\theta}$ on $G\times\mathfrak{g}$ satisfies the following properties: $\mathcal{L}_{\tilde{\theta}}\omega_1=\omega_2$, $\mathcal{L}_{\tilde{\theta}}\omega_2=-\omega_1$, and $\tilde{\theta}$ generates a circle action on $G\times\mathfrak{g}$ that preserves $(I_3,\omega_3,b)$. This circle subgroup acts by rotations on $\mathrm{span}_{\mathbb{R}}\{\omega_1,\omega_2\}$, and the following is (the $\theta$-component of) a moment map for its Hamiltonian action on $(G\times\mathfrak{g},\omega_3)$:
\begin{equation}\label{Equation: First potential}\rho:G\times\mathfrak{g}\rightarrow\mathbb{R},\quad [(T_0,T_1,T_2,T_3)]\mapsto\frac{1}{2}\int_{0}^1\bigg(\langle T_1,T_1\rangle_{\mathfrak{k}}+\langle T_2,T_2\rangle_{\mathfrak{k}}\bigg)dt,\end{equation}
where $[(T_0,T_1,T_2,T_3)]$ denotes the point in the $\Phi^{-1}(0)/\mathcal{G}_0\cong G\times\mathfrak{g}$ represented by $(T_0,T_1,T_2,T_3)\in\Phi^{-1}(0)$ (see \cite[Section 4]{Dancer}). This leads to the following lemma.

\begin{lemma}\label{Lemma: Invariance}
The function $\rho$ is invariant under each of the $K$-actions \eqref{Equation: Left lift} and \eqref{Equation: Right lift} on $G\times\mathfrak{g}$.
\end{lemma}

\begin{proof}
Since $\langle\cdot,\cdot\rangle_{\mathfrak{k}}$ is a $K$-invariant inner product, the function
$$\mathcal{M}\rightarrow\mathbb{R},\quad (T_0,T_1,T_2,T_3)\mapsto\frac{1}{2}\int_{0}^1\bigg(\langle T_1,T_1\rangle_{\mathfrak{k}}+\langle T_2,T_2\rangle_{\mathfrak{k}}\bigg)dt$$ is invariant under the action \eqref{Equation: Gauge group action} of $\mathcal{G}$. This function therefore descends to a $\mathcal{G}/\mathcal{G}_0$-invariant function on the hyperk\"ahler quotient $\mathcal{M}\ssslash\mathcal{G}_0$. The descended function is exactly $\rho$ once we identify $\mathcal{M}\ssslash\mathcal{G}_0$ with $G\times\mathfrak{g}$ via \ref{Equation: Explicit diffeomorphism}. Now recall that the $\mathcal{G}/\mathcal{G}_0$-action on $\mathcal{M}\ssslash\mathcal{G}_0$ corresponds to a $(K\times K)$-action on $G\times\mathfrak{g}$, meaning that $\rho$ is a $(K\times K)$-invariant function on $G\times\mathfrak{g}$. It just remains to recall that the $K$-action \eqref{Equation: Left lift} (resp. \eqref{Equation: Right lift}) is the action of $K=\{e\}\times K\subseteq K\times K$ (resp. $K=K\times\{e\}\subseteq K\times K$). 
\end{proof}

\subsection{The hyperk\"ahler structure on $T^*(G/H)$}\label{Subsection: The hyperkahler structure on the cotangent bundle}
Let $G$ act on $G/H$ via left multiplication, and consider the canonical lift to a Hamiltonian action of $G$ on $T^*(G/H)$. Note also that $(\mathfrak{g}/\mathfrak{h})^*$ is a representation of $H$, and let $G\times_H(\mathfrak{g}/\mathfrak{h})^*$ denote the quotient of $G\times(\mathfrak{g}/\mathfrak{h})^*$ by the following action of $H$: $$h\cdot (g,\phi):=(gh^{-1},h\cdot\phi)\quad h\in G, (g,\phi)\in G\times(\mathfrak{g}/\mathfrak{h})^*.$$ We then have a canonical $G$-equivariant isomorphism
$T^*(G/H)\cong G\times_H(\mathfrak{g}/\mathfrak{h})^*$, where $G$ acts on the latter variety via left multiplication on the first factor. At the same time, the $H$-representation $(\mathfrak{g}/\mathfrak{h})^*$ is canonically isomorphic to the annihilator $\mathfrak{h}^{\perp}\subseteq\mathfrak{g}$ of $\mathfrak{h}$ under the Killing form. We thus have a $G$-equivariant isomorphism \begin{equation}\label{Equation: Cotangent} T^*(G/H)\cong G\times_H\mathfrak{h}^{\perp},\end{equation} with $G\times_H\mathfrak{h}^{\perp}$ defined analogously to $G\times_H(\mathfrak{g}/\mathfrak{h})^*$. 

Now consider the restriction of \eqref{Equation: Right lift} to an action of $H\subseteq G$ on $G\times\mathfrak{g}$, noting that this restricted action is Hamiltonian with respect to $\Omega_L$. The moment map for this $H$-action is obtained by composing the $\mathfrak{g}^*$-valued version of $\phi_R:G\times\mathfrak{g}\rightarrow\mathfrak{g}$ with the projection $\mathfrak{g}^*\rightarrow\mathfrak{h}^*$. It follows that the preimage of $0$ under the new moment map is $G\times\mathfrak{h}^{\perp}\subseteq G\times\mathfrak{g}$. The symplectic quotient of $G\times\mathfrak{g}$ by $H$ is therefore given by
$$(G\times\mathfrak{g})\dslash H=G\times_H\mathfrak{h}^{\perp}.$$ It is straightforward to verify that the induced symplectic structure on $G\times_H\mathfrak{h}^{\perp}$ renders \eqref{Equation: Cotangent} a $G$-equivariant isomorphism of symplectic varieites. It is also straightforward to check that \begin{equation}\label{Equation: Complex moment map}\nu_H:G\times_H\mathfrak{h}^{\perp}\rightarrow\mathfrak{g},\quad [(g,x)]\mapsto \mathrm{Ad}_g(x),\quad (g,x)\in G\times\mathfrak{h}^{\perp}\end{equation} is a moment map for the Hamiltonian action of $G$ on $G\times_H\mathfrak{h}^{\perp}$.

The above-defined holomorphic symplectic structure and Hamiltonian $G$-action on $G\times_H\mathfrak{h}^{\perp}$ turn out to come from a $(G,K)$-hyperk\"ahler variety structure (see Definition \ref{Definition: (G,K)-hyperkahler variety}), which we now discuss. Accordingly, recall that \eqref{Equation: Left lift} and \eqref{Equation: Right lift} define commuting, tri-Hamiltonian actions of $K$ on $G\times\mathfrak{g}$. Let us restrict the latter action to the subgroup $L\subseteq K$ fixed in the introduction to Section \ref{Section: The hyperkahler geometry of T*(G/H)}, and then consider the associated hyperk\"ahler quotient $(G\times\mathfrak{g})\ssslash L$. Note that \eqref{Equation: Left lift} then descends to a tri-Hamiltonian action of $K$ on $(G\times\mathfrak{g})\ssslash L$. At the same time, \eqref{Equation: Hyperkahler and holomorphic symplectic quotients} takes the form of a $K$-equivariant map\begin{equation}\label{Equation: Special case} (G\times\mathfrak{g})\ssslash L\rightarrow (G\times\mathfrak{g})\dslash H=G\times_H\mathfrak{h}^{\perp}.\end{equation}
One can then invoke \cite[Section 2]{Dancer} and/or \cite[Theorem 3.1]{MayrandTG} to deduce the following fact.

\begin{theorem}
The map \eqref{Equation: Special case} is a $K$-equivariant isomorphism of holomorphic symplectic manifolds.
\end{theorem}

Let $(I_{\ell}^H,\omega_{\ell}^H,b^H)$, $\ell=1,2,3$, denote the hyperk\"ahler manifold structure on $G\times_H\mathfrak{h}^{\perp}$ for which \eqref{Equation: Special case} is an isomorphism of hyperk\"ahler manifolds, which by the preceding discussion makes $G\times_H\mathfrak{h}^{\perp}$ into a $(G,K)$-hyperk\"ahler variety. To help investigate this $(G,K)$-hyperk\"ahler structure, we use Lemma \ref{Lemma: Invariance} to see that $\rho$ descends to a $K$-invariant function $\rho^H:(G\times\mathfrak{g})\ssslash L\rightarrow\mathbb{R}$. Note that since \eqref{Equation: Special case} is $K$-equivariant, we may regard $\rho^H$ as a $K$-invariant function on $G\times_H\mathfrak{h}^{\perp}$.          

\begin{proposition}\label{Proposition: Kahler potential}
The function $\rho^H:G\times_H\mathfrak{h}^{\perp}\rightarrow\mathbb{R}$ is a $K$-invariant potential for the K\"ahler manifold $(G\times_H\mathfrak{h}^{\perp},I_1^H,\omega_1^H,b^H)$, i.e. $\omega_1^H=2i\partial\overline{\partial}\rho^H$ for the Dolbeault operators $\partial$ and $\overline{\partial}$ associated with $I_1^H$.
\end{proposition}

\begin{proof}
Let $(\mu_1,\mu_2,\mu_3):G\times\mathfrak{g}\rightarrow(\mathfrak{k}^*)^{\oplus 3}$ denote the hyperk\"ahler moment map for the tri-Hamiltonian $K$-action \eqref{Equation: Right lift}, and let $(\mu_1^H,\mu_2^H,\mu_3^H):G\times\mathfrak{g}\rightarrow(\mathfrak{l}^*)^{\oplus 3}$ be the induced hyperk\"ahler moment map for the action of $L\subseteq K$. Consider the action of $\operatorname{SO}_3(\mathbb{R})$ on $G\times\mathfrak{g}$, recalling our description of a specific subgroup $S^1\subseteq\operatorname{SO}_3(\mathbb{R})$ and its action on $G\times\mathfrak{g}$ (see \ref{Subsection: Kronheimer's structure}). This description implies that $S^1$ preserves $\mu_{3}$ and acts by rotations on $\mathrm{span}_{\mathbb{R}}\{\mu_{1},\mu_{2}\}$. We conclude that $S^1$ preserves $\mu_3^H$ and acts by rotations on $\mathrm{span}_{\mathbb{R}}\{\mu_{1}^H,\mu_{2}^H\}$, so that the submanifold $(\mu_{1}^H)^{-1}(0)\cap(\mu_{2}^H)^{-1}(0)\cap(\mu_{3}^H)^{-1}(0)\subseteq G\times\mathfrak{g}$ is necessarily $S^1$-invariant. Observe that the actions of $S^1$ and $L$ on this submanifold commute, owing to the fact that the action of $\operatorname{SO}_3(\mathbb{R})$ on $G\times\mathfrak{g}$ commutes with the $K$-action \eqref{Equation: Right lift}. The quotient $$\big((\mu_{1}^H)^{-1}(0)\cap(\mu_{2}^H)^{-1}(0)\cap(\mu_{3}^H)^{-1}(0)\big)/L=(G\times\mathfrak{g})\ssslash L$$ therefore carries a residual $S^1$-action, so that we may use the hyperk\"ahler isomorphism \eqref{Equation: Special case} to equip $G\times_H\mathfrak{h}^{\perp}$ with a corresponding $S^1$-action. The relations \eqref{Equation: Symplectic form relations} then imply that $S^1$ preserves $\omega_3^H$. Now consider the element $\theta\in\mathfrak{so}_3(\mathbb{R})$ discussed in \ref{Subsection: Kronheimer's structure}, recalling that $\rho$ is the $\theta$-component of a moment map for the $S^1$-action on $G\times\mathfrak{g}$. It is then straightforward to check that $\rho^H$ is the $\theta$-component of a moment map for the $S^1$-action that preserves $\omega_3^H$. Note also that the identities $\mathcal{L}_{\tilde{\theta}}\omega_1=\omega_2$ and $\mathcal{L}_{\tilde{\theta}}\omega_2=-\omega_1$ give $$\mathcal{L}_{\tilde{\tilde{\theta}}}\omega_1^H=\omega_2^H\quad \text{ and }\quad \mathcal{L}_{\tilde{\tilde{\theta}}}\omega_2^H=-\omega_1^H,$$ where $\tilde{\tilde{\theta}}$ is the fundamental vector field on $G\times_H\mathfrak{h}^{\perp}$ associated to $\theta$. These last two sentences give exactly the ingredients needed to reproduce a calculation from \cite[Section 3(E)]{Hitchin}, to the effect that $\rho^H$ is a K\"ahler potential for $(I_1^H,\omega_1^H,b^H)$. 
\end{proof}

\section{The hyperk\"ahler slice construction}\label{Section: The hyperkahler slice construction}

\subsection{The slice as a symplectic variety}\label{Subsection: The slice as a symplectic variety}
Recall the notation established in the introduction to Section \ref{Section: The hyperkahler geometry of T*(G/H)}, and let $$\adj:\mathfrak{g}\rightarrow\mathfrak{gl}(\mathfrak{g}),\quad x\mapsto\adj_x,\quad x\in\mathfrak{g}$$ denote the adjoint representation of $\mathfrak{g}$. One calls $\tau=(\xi,h,\eta)\in\mathfrak{g}^{\oplus 3}$ an $\mathfrak{sl}_2$-\textit{triple} if $[\xi,\eta]=h$, $[h,\xi]=2\xi$, and $[h,\eta]=-2\eta$, in which case there is an associated \textit{Slodowy slice} $$S_{\tau}:=\xi+\mathrm{ker}(\mathrm{ad}_{\eta})\subseteq\mathfrak{g}.$$ We will make extensive use of the affine variety $G\times S_{\tau}$, some geometric features of which we now develop. 

Consider the isomorphisms $T^*G\cong G\times\mathfrak{g}^*\cong G\times\mathfrak{g}$ induced by the right trivialization of $T^*G$ and the Killing form. The symplectic form on $T^*G$ thereby corresponds to such a form $\Omega_R$ on $G\times\mathfrak{g}$, described as follows on the tangent space $T_{(g,x)}(G\times\mathfrak{g})=T_gG\oplus\mathfrak{g}$ (see \cite[Section 5, Equation (14R)]{Marsden}:
\begin{equation}\label{Equation: Right trivialized form}(\Omega_R)_{(g,x)}\bigg((d_eR_g(y_1),z_1),(d_eR_g(y_2),z_2)\bigg)=\langle y_1,z_2\rangle-\langle y_2,z_1\rangle-\langle x,[y_1,y_2]\rangle\end{equation}
for all $y_1,y_2,z_1,z_2\in\mathfrak{g}$, where $R_g:G\rightarrow G$ is right multiplication by $g$, $d_eR_g$ is its differential at $e$. It turns out that $G\times S_{\tau}$ is a symplectic subvariety of $(G\times\mathfrak{g},\Omega_R)$. The $G$-action $$g\cdot (h,x)=(hg^{-1},x),\quad g\in G,\text{ }(h,x)\in G\times S_{\tau}$$ is then Hamiltonian and $$\mu_{\tau}:G\times S_{\tau}\rightarrow\mathfrak{g},\quad (g,x)\mapsto -\mathrm{Ad}_{g^{-1}}(x),\quad (g,x)\in G\times S_{\tau}$$ is a moment map.

\begin{remark}
Bielawski's paper \cite{Bielawski} uses $\Omega_R$ to realize $G\times S_{\tau}$ as a symplectic subvariety $G\times\mathfrak{g}$, as opposed to using the other symplectic form $\Omega_L$ (see \eqref{Equation: Left structure}). It is for the sake of consistency with Bielawski's work that we are using the same convention. However, this is the only case in which we use $\Omega_R$ preferentially to $\Omega_L$.        
\end{remark}

Now let $X$ be a symplectic variety endowed with a Hamiltonian $G$-action and moment map $\mu:X\rightarrow\mathfrak{g}$. The diagonal action of $G$ on $X\times (G\times S_{\tau})$ is then Hamiltonian and admits a moment map of $$\tilde{\mu}:X\times (G\times S_{\tau})\rightarrow\mathfrak{g},\quad (x,(g,y))\mapsto \mu(x)+\mu_{\tau}(g,y),\quad x\in X,\text{ }(g,y)\in G\times S_{\tau}.$$ Noting that this diagonal action is free, one has the symplectic quotient $$(X\times (G\times S_{\tau}))\dslash G=\tilde{\mu}^{-1}(0)/G.$$

\begin{proposition}\label{Proposition: Variety isomorphism}
Let $(X,\omega)$ be a symplectic variety on which $G$ acts in a Hamiltonian fashion with moment map $\mu:X\rightarrow\mathfrak{g}$, and let $\tau$ be an $\mathfrak{sl}_2$-triple. The following statements then hold. 
\begin{itemize}
\item[(i)] There is a canonical isomorphism of affine varieties
$\mu^{-1}(S_{\tau})\cong (X\times (G\times S_{\tau}))\dslash G$.
\item[(ii)] Under the isomorphism from \textnormal{(i)}, the symplectic form on $(X\times (G\times S_{\tau}))\dslash G$ corresponds to the restriction of $\omega$ to $\mu^{-1}(S_{\tau})$.
\item[(iii)] $\mu^{-1}(S_{\tau})$ is a symplectic subvariety of $X$.
\end{itemize}
\end{proposition}

\begin{proof}
To prove (i), note that $(x,(e,\mu(x)))\in \tilde{\mu}^{-1}(0)$ for all $x\in\mu^{-1}(S_{\tau})$. We may therefore consider the morphism \begin{equation}\label{Equation: Lifted map}\varphi:\mu^{-1}(S_{\tau})\rightarrow \tilde{\mu}^{-1}(0), \quad x\mapsto (x,(e,\mu(x))),\quad x\in \mu^{-1}(S_{\tau}),\end{equation} and its composition with the quotient map $\pi:\tilde{\mu}^{-1}(0)\rightarrow \tilde{\mu}^{-1}(0)/G=(X\times (G\times S_{\tau}))\dslash G$, i.e.
$$\overline{\varphi}:\mu^{-1}(S_{\tau})\rightarrow (X\times (G\times S_{\tau}))\dslash G, \quad x\mapsto [(x,(e,\mu(x)))],\quad x\in \mu^{-1}(S_{\tau}).$$
At the same time, it is straightforward to check that $g\cdot x\in\mu^{-1}(S_{\tau})$ for all $(x,(g,y))\in\tilde{\mu}^{-1}(0)$. We thus have the morphism $$\psi:\tilde{\mu}^{-1}(0)\rightarrow\mu^{-1}(S_{\tau}),\quad (x,(g,y))\mapsto g\cdot x,\quad (x,(g,y))\in\tilde{\mu}^{-1}(0),$$ which is easily seen to be $G$-invariant. It follows that $\psi$ descends to the quotient $\tilde{\mu}^{-1}(0)/G=(X\times (G\times S_{\tau}))\dslash G$, thereby giving a morphism $\overline{\psi}:(X\times (G\times S_{\tau}))\dslash G\rightarrow\mu^{-1}(S_{\tau})$. Furthermore, it is a straightforward calculation that $\overline{\varphi}$ and $\overline{\psi}$ are inverses. This proves (i).

In preparation for (ii), let $\overline{\omega}$ denote the symplectic form on $(X\times (G\times S_{\tau}))\dslash G$ and consider the inclusions $j:G\times S_{\tau}\rightarrow G\times\mathfrak{g}$ and $k:\tilde{\mu}^{-1}(0)\rightarrow X\times (G\times S_{\tau})$. Note that $\pi^*(\overline{\omega})$ is the restriction to $\tilde{\mu}^{-1}(0)$ of the symplectic form on $X\times (G\times S_{\tau})$. This last symplectic form is $\omega\oplus j^*(\Omega_R)$, so that we have \begin{equation}\label{Equation: Pullback identity}\pi^*(\overline{\omega})=k^*(\omega\oplus j^*(\Omega_R)).\end{equation}

Our objective is to prove that $\overline{\varphi}^*(\overline{\omega})=\ell^*(\omega)$, where $\ell:\mu^{-1}(S_{\tau})\rightarrow X$ is the inclusion. Accordingly, note that
\begin{align*}
\overline{\varphi}^*(\overline{\omega}) & =\varphi^*(\pi^*(\overline{\varphi}))\hspace{200pt}[\text{since }\overline{\varphi}=\pi\circ\varphi]\\
& =(k\circ\varphi)^*(\omega\oplus j^*(\Omega_R))\hspace{200pt}[\text{by }\eqref{Equation: Pullback identity}].
\end{align*}
It follows that
\begin{equation}\label{Equation: First pullback}\big(\overline{\varphi}^*(\overline{\omega})\big)_x(v_1,v_2)=\big(\omega_x\oplus (\Omega_R)_{(e,\mu(x))}\big)(d_x\varphi(v_1),d_x\varphi(v_2))
\end{equation}
for all $x\in \mu^{-1}(S_{\tau})$ and $v_1,v_2\in T_x(\mu^{-1}(S_{\tau}))$. At the same time, \eqref{Equation: Lifted map} implies the identity $$d_x\varphi(v_i)=(v_i,(0,d_x\mu(v_i)))$$ for $i=1,2$ in the tangent space $$T_{(x,(e,\mu(x)))}(\tilde{\mu}^{-1}(0))\subseteq T_{(x,(e,\mu(x)))}(X\times (G\times S_{\tau}))=T_xX\oplus (\mathfrak{g}\oplus T_{\mu(x)}S_{\tau}).$$ By incorporating this into \eqref{Equation: First pullback}, we obtain
\begin{align*}\big(\overline{\varphi}^*(\overline{\omega})\big)_x(v_1,v_2) & = \omega_x(v_1,v_2)+(\Omega_R)_{(e,\mu(x))}\big((0,d_x\mu(v_1)),(0,d_x\mu(v_2)\big)\\
& =  \omega_x(v_1,v_2) \hspace{200pt}[\text{by }\eqref{Equation: Right trivialized form}].
\end{align*}
We conclude that $\overline{\varphi}^*(\overline{\omega})=\ell^*(\omega)$, proving (ii).

It remains only to prove (iii), i.e. the claim that $\ell^*(\omega)$ is non-degenerate. However, this follows immediately from (i), (ii), and the fact that $\overline{\omega}$ is non-degenerate. 
\end{proof}  

\subsection{Bielawski's construction}\label{Subsection: Bielawski's construction}
We now review the pertinent hyperk\"ahler-geometric features of $\mu^{-1}(S_{\tau})$, which are largely due to Bielawski's work \cite{Bielawski}. The following $(G,K)$-hyperk\"ahler variety will play an essential role.

\begin{theorem}
If $\tau$ is an $\mathfrak{sl}_2$-triple, then $G\times S_{\tau}$ is canonically a $(G,K)$-hyperk\"ahler variety. The Hamiltonian $G$-action and underlying holomorphic symplectic structure on $G\times S_{\tau}$ associated with this $(G,K)$-hyperk\"ahler structure are precisely those described in \ref{Subsection: The slice as a symplectic variety}. 
\end{theorem}

Now let $X$ be any $(G,K)$-hyperk\"ahler variety. Given an $\mathfrak{sl}_2$-triple $\tau$, note that product manifold $X\times (G\times S_{\tau})$ is naturally hyperk\"ahler and carries a free, diagonal $G$-action. It is then not difficult to check that $X\times (G\times S_{\tau})$ is a $(G,K)$-hyperk\"ahler variety, with underlying holomorphic symplectic structure equal to the natural product holomorphic symplectic structure on $X\times (G\times S_{\tau})$. With this in mind, we can define \textit{hyperk\"ahler slices} as follows. 

\begin{definition}\label{Definition: Hyperkahler slice}
Given a $(G,K)$-hyperk\"ahler variety $X$ and an $\mathfrak{sl}_2$-triple $\tau$, we refer to $(X\times(G\times S_{\tau}))\ssslash K$ as the \textit{hyperk\"ahler slice} for $X$ and $\tau$.  
\end{definition}

This construction can be used to produce a number of well-studied hyperk\"ahler manifolds, some of which are mentioned in the introduction of \cite{BielawskiComplex}. For several of these examples, there is a particularly concrete description of the underlying holomorphic symplectic manifold. Indeed, let $X$ and $\tau$ be as described in the definition above. Note that \eqref{Equation: Hyperkahler and holomorphic symplectic quotients} manifests as a map
\begin{equation}\label{Equation: Second special case}(X\times(G\times S_{\tau}))\ssslash K\rightarrow (X\times(G\times S_{\tau}))\dslash G,\end{equation} which features in the following rephrased version of \cite[Theorem 1]{Bielawski}.
 
\begin{theorem}[Bielawski]\label{Theorem: Hyperkahler slice theorem}
Let $\tau$ be an $\mathfrak{sl}_2$-triple, and let $(X,(I_{\ell},\omega_{\ell},b)_{\ell=1}^3)$ be a $(G,K)$-hyperk\"ahler variety with complex moment map $\mu:X\rightarrow\mathfrak{g}$. Consider the map \begin{equation}\label{Equation: Holomorphic symplectic isomorphism}(X\times(G\times S_{\tau}))\ssslash K \rightarrow \mu^{-1}(S_{\tau})\end{equation} obtained by composing \eqref{Equation: Second special case} with the isomorphism $(X\times(G\times S_{\tau}))\dslash G\xrightarrow{\cong}\mu^{-1}(S_{\tau})$ from Proposition \ref{Proposition: Variety isomorphism}(i). If the K\"ahler manifold $(X,I_{1},\omega_{1},b)$ has a $K$-invariant potential that is bounded from below on each $G$-orbit, then \eqref{Equation: Holomorphic symplectic isomorphism} is an isomorphism of holomorphic symplectic manifolds. 
\end{theorem}

\begin{remark}
Bielawski speaks of hyperk\"ahler slices only when the hypotheses of Theorem \ref{Theorem: Hyperkahler slice theorem} are satisfied (see \cite[Section 1]{BielawskiComplex}). He then defines a hyperk\"ahler slice to be a hyperk\"ahler manifold of the form $\mu^{-1}(S_{\tau})$, where $\mu^{-1}(S_{\tau})$ is equipped with the hyperk\"ahler structure induced through the isomorphism \eqref{Equation: Holomorphic symplectic isomorphism}. In particular, Definition \ref{Definition: Hyperkahler slice} mildly generalizes Bielawski's original notion.  
\end{remark}

Let us briefly consider the hyperk\"ahler slice construction for $(G,K)$-hyperk\"ahler varieties of the form $(G\times_H\mathfrak{h}^{\perp},(I_{\ell}^H,\omega_{\ell}^H,b^H)_{\ell=1}^3)$, as introduced in \ref{Subsection: The hyperkahler structure on the cotangent bundle}. Accordingly, recall the notation adopted in \ref{Subsection: The hyperkahler structure on the cotangent bundle}. The function $\rho^H$ is bounded from below on all of $G\times_H\mathfrak{h}^{\perp}$ (see \eqref{Equation: First potential}), while we recall that $\rho^H$ is a $K$-invariant potential for the K\"ahler manifold $(G\times_H\mathfrak{h}^{\perp},(I_1^H,\omega_1^H,b^H))$ (see Proposition \ref{Proposition: Kahler potential}). It then follows from Theorem \ref{Theorem: Hyperkahler slice theorem} that
\begin{equation}\label{Equation: Second holomorphic symplectic isomorphism}\left((G\times_H\mathfrak{h}^{\perp})\times (G\times S_{\tau})\right)\ssslash K\cong \nu_H^{-1}(S_{\tau})\end{equation} as holomorphic symplectic manifolds for all $\mathfrak{sl}_2$-triples $\tau$ in $\mathfrak{g}$. We exploit this fact in what follows.      

\subsection{The regular Slodowy slice}

Recall that $\dim(\mathrm{ker}(\adj_x))\geq r$ for all $x\in\mathfrak{g}$, and that $x$ is called \textit{regular} if equality holds. Let $\mathfrak{g}_{\text{reg}}\subseteq\mathfrak{g}$ denote the set of all regular elements, which is known to be a $G$-invariant, open, dense subvariety of $\mathfrak{g}$. This leads to the notion of a \textit{regular} $\mathfrak{sl}_2$-triple, i.e. an $\mathfrak{sl}_2$-triple $\tau=(\xi,h,\eta)$ in $\mathfrak{g}$ for which $\xi\in\mathfrak{g}_{\text{reg}}$. Fix one such triple $\tau$ for the duration of this paper, and let $S_{\text{reg}}:=S_{\tau}$ denote the associated Slodowy slice. The slice $S_{\text{reg}}$ is known to be contained in $\mathfrak{g}_{\text{reg}}$, and to be a fundamental domain for the action of $G$ on $\mathfrak{g}_{\text{reg}}$ (see \cite[Theorem 8]{Kostant}). Note that this last sentence may be rephrased as follows: $x\in\mathfrak{g}$ belongs to $\mathfrak{g}_{\text{reg}}$ if and only if $x$ is $G$-conjugate to a point in $S_{\text{reg}}$, in which case $x$ is $G$-conjugate to a unique point in $S_{\text{reg}}$. 

As discussed in the \ref{Subsection: Description of results}, we wish to study the emptiness problem for hyperk\"ahler slices of the form $\left((G\times_H\mathfrak{h}^{\perp})\times (G\times S_{\text{reg}})\right)\ssslash K$. The following result is a crucial first step.  

\begin{proposition}\label{Proposition: Non-empty slice}
The hyperk\"ahler slice $\left((G\times_H\mathfrak{h}^{\perp})\times (G\times S_{\emph{reg}})\right)\ssslash K$ is non-empty if and only if $\mathfrak{h}^{\perp}\cap\mathfrak{g}_{\emph{reg}}\neq\emptyset$. 
\end{proposition}

\begin{proof}
Using \eqref{Equation: Second holomorphic symplectic isomorphism}, we conclude that $\left((G\times_H\mathfrak{h}^{\perp})\times (G\times S_{\text{reg}})\right)\ssslash K\neq\emptyset$ if and only if the image of $\nu_H$ meets $S_{\text{reg}}$. This image is precisely $G\cdot\mathfrak{h}^{\perp}\subseteq\mathfrak{g}$ (see \eqref{Equation: Complex moment map}), reducing our task to one of proving that $G\cdot\mathfrak{h}^{\perp}\cap S_{\text{reg}}\neq\emptyset$ if and only if $\mathfrak{h}^{\perp}\cap\mathfrak{g}_{\text{reg}}\neq\emptyset$. To prove this, we simply appeal to the discussion of $S_{\text{reg}}$ above and note that $x\in\mathfrak{h}^{\perp}$ belongs to $\mathfrak{g}_{\text{reg}}$ if and only if $x$ is $G$-conjugate to a point in $S_{\text{reg}}$. 
\end{proof}

\section{The spherical geometry of $G/H$}\label{Section: The spherical geometry of T*(G/H)}

\subsection{The image of the moment map.}\label{subsection: image of moment map}

Let us continue with the notation set in the introduction of Section \ref{Section: The hyperkahler geometry of T*(G/H)}. Choose opposite Borel subgroups $B,B_{-}\subseteq G$, declaring the former to be the positive Borel and the latter to be the negative Borel. It follows that $T:=B\cap B_{-}$ is a maximal torus of $G$, and we shall let $\mathfrak{b}$, $\mathfrak{b}_{-}$, and $\mathfrak{t}$ denote the Lie algebras of $B$, $B_{-}$, and $T$, respectively. We thus have a weight lattice $\Lambda\subseteq\mathfrak{t}^*$ and canonical group isomorphisms $\Lambda\cong\mathrm{Hom}(T,\mathbb{C}^{\times})\cong\mathrm{Hom}(B,\mathbb{C}^{\times})$, where $\mathrm{Hom}$ is taken in the category of algebraic groups. We also have sets of roots $\Delta\subseteq\Lambda$, positive roots $\Delta_{+}\subseteq\Delta$, negative roots $\Delta_{-}\subseteq\Delta$, and simple roots $\Pi\subseteq\Delta_{+}$. Note that by definition
$$\mathfrak{b}=\mathfrak{t}\oplus\bigoplus_{\alpha\in\Delta_{+}}\mathfrak{g}_{\alpha}\quad \text{ and }\quad \mathfrak{b}_{-}=\mathfrak{t}\oplus\bigoplus_{\alpha\in\Delta_{-}}\mathfrak{g}_{\alpha},$$ where $\mathfrak{g}_{\alpha}$ is the root space associated to $\alpha\in\Delta$.
 
We now establish two important conventions. To this end, recall the isomorphism \eqref{Equation: Killing isomorphism} between the adjoint and coadjoint representations of $G$. Our first convention is to use $(\cdot)^{\vee}$ for both \eqref{Equation: Killing isomorphism} and its inverse, so that the inverse will presented as  $$\mathfrak{g}^*\xrightarrow{\cong}\mathfrak{g},\quad \phi\mapsto\phi^{\vee},\quad\phi\in\mathfrak{g}^*.$$ As for our second convention, note that the map $\mathfrak{g}^*\rightarrow\mathfrak{t}^*$ restricts to an isomorphism from the image of $\mathfrak{t}$ under \eqref{Equation: Killing isomorphism} to $\mathfrak{t}^*$. We will use this isomorphism to regard $\mathfrak{t}^*$ as belonging to $\mathfrak{g}^*$.

Now let $Y$ be a smooth, irreducible $G$-variety having field of rational functions $\mathbb{C}(Y)$, noting that $\mathbb{C}(Y)$ is then a $G$-module. A non-zero $f\in\mathbb{C}(Y)$ is called a $B$\textit{-semi-invariant rational function of weight} $\lambda\in\Lambda$ if 
$b\cdot f=\lambda(b)f$ for all $b\in B$. Those $\lambda$ admitting such an $f$ form the \textit{weight lattice of} $Y$, i.e.
$$\Lambda_Y:=\{\lambda\in\Lambda:\exists\text{ a }B\text{-semi-invariant rational function on $Y$ of weight }\lambda\}.$$
The weight lattice of $Y$ can also be viewed as the character lattice of a quotient of $T$, once we appeal to Knop's local structure theorem \cite[Theorem 4.7]{TimashevBook}. This theorem gives a parabolic subgroup $P\subseteq G$ that contains $B$, has a Levi decomposition $P=P_{u}L$ with $T\subseteq L$, and satisfies the following property: there exists a locally closed affine $P$-stable subvariety $Z\subseteq Y$ such that $P_{u}\times Z\to Y$ maps surjectively onto an open affine subset $Y_{0}$ of Y. One also has $[L,L]\subseteq L_0\subseteq L$, where $L_0$ is the kernel of the $L$-action on $Z$.
The quotient $A_{Y}:=L/L_{0}$ is a torus that acts freely on $Z$, and there exists an affine variety $C$ with a trivial $L$-action such that $Z\cong A_Y\times C$ as $L$-varieties. 
It follows that $\Lambda_{Y}=\Hom(A_{Y},\bbC^{\times})$.

The subspace $\laa^{*}_{Y}:=\Lambda_{Y}\otimes_{\bbZ}\bbC\subseteq\mathfrak{t}^*$ is sometimes called the \textit{Cartan space} of the $G$-variety $Y$.
Let $\Lambda_Y^{\vee}\subseteq\mathfrak{t}$ and $\laa_{Y}\subseteq\mathfrak{t}$ denote the preimage and image of $\Lambda_Y$ and $\laa^{*}_{Y}$ under \eqref{Equation: Killing isomorphism}, respectively, noting that \begin{equation}\label{Equation: Subtorus} \widetilde{A}_Y:=\Lambda_Y^{\vee}\otimes_{\mathbb{Z}}\mathbb{C}^{\times}\end{equation} is a subtorus of $T$ with Lie algebra $\laa_{Y}$. We shall also refer to $\mathfrak{a}_Y$ as the Cartan space of $Y$.

\begin{example}\label{Example: Cartan space of G/T}
In what follows, we compute the Cartan space of $G/T$. Let $\Lambda_{+}\subseteq\mathfrak{t}^*$ denote the set of dominant weights of $G$, and let $V_{\lambda}$ be the irreducible $G$-module of highest weight $\lambda\in\Lambda_{+}$. Recall the following classical fact about $\bbC[G/T]$, the coordinate ring of $G/T$:
$$\bbC[G/T]\cong\bigoplus_{\Lambda\in\Lambda_{+}}(V_{\lambda}^*)^{\oplus d_{\lambda}}$$
as $G$-modules, where $d_{\lambda}:=\dim((V_{\lambda})^T)$ and $(V_{\lambda})^T$ is the subspace of $T$-fixed vectors in $V_{\lambda}$. Note that $d_{\lambda}\neq 0$ if and only if $\lambda$ lies in the root lattice $Q\subseteq\mathfrak{t}^*$. Note also that $(V_{\lambda})^*\cong V_{-w_0\lambda}$ as $G$-modules, where $w_0$ is the longest element of the Weyl group $W:=N_G(T)/T$. It follows that for $\lambda\in\Lambda_{+}$, $V_{\lambda}$ is an irreducible summand of $\bbC[G/T]$ if and only if $\lambda\in -w_0(\Lambda_{+}\cap Q) = \Lambda_{+}\cap Q$. Since $\bbC[G/T]$ is a $G$-submodule of $\bbC(G/T)$, this implies that $\Lambda_{+}\cap Q$ is contained in $\Lambda_{G/T}$. Now observe that $\Lambda_{+}\cap Q$ generates $\lat^*$ over $\mathbb{C}$, yielding $\laa_{G/T}^*=\Lambda_{G/T}\otimes_{\mathbb{Z}}\mathbb{C}=\mathfrak{t}^*$. We also conclude that $\laa_{G/T}=\lat$.
\end{example}

We now recall a key geometric feature of the Cartan space construction. Let $Y$ be any smooth, irreducible $G$-variety and consider the canonical lift of the $G$-action on $Y$ to a $G$-action on $T^*Y$. The latter action is Hamiltonian with respect to the standard symplectic form on $T^*Y$, and there is a distinguished moment map $\mu_Y:T^*Y\rightarrow\mathfrak{g}$. Lemma 3.1 and Corollary 3.3 from \cite{KnopAsymptotic} then combine to give the following equality of closures in $\mathfrak{g}$.

\begin{theorem}[Knop]\label{Theorem: Knop's theorem}
If $Y$ is a smooth, irreducible, quasi-affine $G$-variety, then $\overline{\mu_Y(T^*Y)}=\overline{G\cdot\mathfrak{a}_Y}$.
\end{theorem}

\subsection{$\mathfrak{a}$-regularity}\label{subsection: The cotangent bundle of a homogeneous space}

Recall the notation set in the introduction to Section \ref{Section: The hyperkahler geometry of T*(G/H)}, which we now use together with the notation of \ref{subsection: image of moment map}. It is then not difficult to prove that $\mathfrak{a}_{G/H}$ depends only on the pair $(\lag,\lah)$. For this reason, we set $\laa(\lag,\lah)^{*}:=\laa_{G/H}^*$ and $\laa(\lag,\lah):=\laa_{G/H}$. We will sometimes denote $\laa(\lag,\lah)$ (resp. $\laa(\lag,\lah)^*$) by $\mathfrak{a}$ (resp. $\laa^*$) when the underlying pair $(\lag,\lah)$ is clear from context. 

\begin{definition}\label{Definition: a-regular}
We say that the pair $(G,H)$ or the corresponding pair $(\lag,\lah)$ of Lie algebras is $\laa$-\textit{regular} if $\laa(\lag,\lah)$ contains a regular element of $\lag$.
\end{definition}

We now give a few characterizations of $\mathfrak{a}$-regularity. In what follows, $\widetilde{A}_{G/H}$ is the subtorus of $T$ defined by setting $Y=G/H$ in \eqref{Equation: Subtorus} and $Z_G(\widetilde{A}_{G/H})$ consists of all $g\in G$ that commute with every element of $\widetilde{A}_{G/H}$. We also let $Z_G(\mathfrak{a})$ be the subgroup of all $g\in G$ that fix $\mathfrak{a}$ pointwise, and we let $\laz_{\mathfrak{g}}(\mathfrak{a})$ be the subspace of all $x\in\mathfrak{g}$ that commute with every element of $\mathfrak{a}$.

\begin{proposition}\label{prop: regular element}
With all notation as described above, the following conditions are equivalent.
\begin{itemize}
\item[(i)] $(G,H)$ is $\laa$-regular;
\item[(ii)] $\mathfrak{h}^{\perp}\cap\mathfrak{g}_{\emph{reg}}\neq\emptyset$; 
\item[(iii)] $Z_{G}(\mathfrak{a})=T$.
\end{itemize}
\end{proposition}

\begin{proof}
We begin by proving that $\mathfrak{h}^{\perp}\cap\mathfrak{g}_{\text{reg}}\neq\emptyset$ if and only if $(G,H)$ is $\mathfrak{a}$-regular. To show the forward implication, assume that $\mathfrak{h}^{\perp}\cap\mathfrak{g}_{\text{reg}}\neq\emptyset$. Identifying $T^*(G/H)$ with $G\times_H\mathfrak{h}^{\perp}$ and recalling the moment map $\nu_H$ (see \eqref{Equation: Complex moment map}), Theorem \ref{Theorem: Knop's theorem} implies that
$$\overline{\nu_H(G\times_H\mathfrak{h}^{\perp})}=\overline{G\cdot\mathfrak{a}}.$$ This amounts to the statement that
$$\overline{G\cdot\mathfrak{h}^{\perp}}=\overline{G\cdot\mathfrak{a}}.$$
Since $\mathfrak{h}^{\perp}\cap\mathfrak{g}_{\text{reg}}\neq\emptyset$ by hypothesis, we must have $\overline{G\cdot\mathfrak{a}}\cap\mathfrak{g}_{\text{reg}}\neq\emptyset$. Note also that $G\cdot\mathfrak{a}$ is a constructible subset of $\mathfrak{g}$, so that $G\cdot \mathfrak{a}$ intersects every non-empty open subset of $\overline{G\cdot\mathfrak{a}}$. These last two sentences imply that $G\cdot \mathfrak{a}\cap\mathfrak{g}_{\text{reg}}\neq\emptyset$, which is equivalent to $\mathfrak{a}\cap\mathfrak{g}_{\text{reg}}\neq\emptyset$. We conclude that $(G,H)$ is $\mathfrak{a}$-regular. In an analogous way, one argues that $(G,H)$ being $\mathfrak{a}$-regular implies $\mathfrak{h}^{\perp}\cap\mathfrak{g}_{\text{reg}}\neq\emptyset$.

We are reduced to establishing that $(G,H)$ is $\mathfrak{a}$-regular if and only if $Z_{G}(\mathfrak{a})=T$. Accordingly, recall that an element of $\mathfrak{t}$ is regular if and only if it does not lie on any root hyperplane. It follows that $(G,H)$ is not $\mathfrak{a}$-regular if and only if $\mathfrak{a}$ belongs to the union of all root hyperplanes. Since $\mathfrak{a}$ is irreducible, this is equivalent to $\mathfrak{a}$ being contained in a particular root hyperplane, i.e. $\mathfrak{a}\subseteq\mathrm{ker}(\alpha)$ for some $\alpha\in\Delta$. This holds if and only if $\mathfrak{g}_{\alpha}\subseteq \laz_{\mathfrak{g}}(\mathfrak{a})$ for some $\alpha\in\Delta$. Now note that $\laz_{\mathfrak{g}}(\mathfrak{a})$ is a $T$-invariant subspace of $\mathfrak{g}$ containing $\mathfrak{t}$, meaning that $$\laz_{\mathfrak{g}}(\mathfrak{a})=\mathfrak{t}\oplus\bigoplus_{\alpha\in S}\mathfrak{g}_{\alpha}$$ for some subset $S\subseteq\Delta$. It follows that $\mathfrak{g}_{\alpha}\subseteq \laz_{\mathfrak{g}}(\mathfrak{a})$ for some $\alpha\in\Delta$ if and only if $\laz_{\mathfrak{g}}(\mathfrak{a})\neq\mathfrak{t}$. The second of these conditions is equivalent to having $Z_G(\mathfrak{a})\neq T$, if one knows $Z_G(\mathfrak{a})$ to be connected and have a Lie algebra of $\laz_{\mathfrak{g}}(\mathfrak{a})$. Connectedness follows from the observation that $Z_G(\mathfrak{a})=Z_G(\widetilde{A}_{G/H})$ (see \cite[Theorem 24.4.8]{Tauvel}), together with the fact that centralizers of tori are connected (see \cite[Proposition 28.3.1]{Tauvel}).  At the same time, it is clear that $\laz_{\mathfrak{g}}(\mathfrak{a})$ is the Lie algebra of $Z_G(\mathfrak{a})$ (cf. \cite[Proposition 24.3.6]{Tauvel}). This completes the proof. 
\end{proof}

Let $H$ act on a complex algebraic variety $X$. A subgroup $\widetilde{H}\subseteq H$ is called a \textit{generic stabilizer} for this action if there exists a non-empty open dense subset $U\subseteq X$ with the following property: the $H$-stabilizer of every $x\in U$ is conjugate to $\widetilde{H}$. A generic stabilizer is known to exist if $X$ is a linear representation of $H$ \cite{MR0294336}. We therefore have a generic stabilizer for the $H$-action on $\lah^{\perp}$, and we denote it by $H_{*}$. This group is known to be reductive (see \cite[Theorem 9.1]{TimashevBook}).

\begin{remark}\label{Remark: Generic stabilizer} A generic stabilizer is unique up to conjugation, meaning that $H_{*}$ more appropriately denotes a conjugacy class of subgroups in $H$. However, we shall always take $H_{*}$ to be a fixed subgroup in this conjugacy class.
\end{remark} 

Now recall our discussion of the the local structure theorem for a smooth, irreducible $G$-variety $Y$, as well as the notation introduced in that context (see \ref{subsection: image of moment map}). If $Y=G/H$, then the group $L_0$ turns out to be precisely $H_{*}$ (see \cite[Section 8]{KnopWeyl}).

\begin{corollary}\label{cor: gen stab}
The pair $(G,H)$ is $\laa$-regular if and only if the connected component of the identity in $H_{*}$ is abelian. 
\end{corollary}

\begin{proof}
Proposition \ref{prop: regular element} and the fact that $H_{*}=L_0$ reduce our task to one of proving that $Z_G(\laa)=T$ if and only if the identity component in $L_0$ is abelian. To this end, consider \cite[Definition 8.13]{TimashevBook} and \cite[Proposition 8.14]{TimashevBook}. Since $G/H$ is an affine variety, these two statements imply that $L=Z_G(\laa)$. Our task is therefore to prove that $L=T$ if and only if the identity component in $L_0$ is abelian. The forward implication follows immediately from the inclusion $L_0\subseteq L$, so that we only need to verify the opposite implication.

Note that $L$ is a Levi factor of a parabolic subgroup of $G$, as discussed in \ref{subsection: image of moment map}. This means that $L$ is connected and reductive, forcing the derived subgroup $[L,L]$ to be connected as well. The inclusion $[L,L]\subseteq L_0$ thus shows $[L,L]$ to be contained in the identity component in $L_0$. If we now assume that this component is abelian, then $[L,L]$ must also be abelian. It follows that $L$ is itself abelian. Together with the inclusion $T\subseteq L$ (see \ref{subsection: image of moment map}) and the fact that $L$ is a connected, reductive subgroup of $G$, this last sentence implies that $L=T$. The proof is complete.     
\end{proof}

Corollary \ref{cor: gen stab} can be used to easily assess $\mathfrak{a}$-regularity in several examples. To see this, we note that \cite{MR0376965} fully describes the $H$-representation $\mathfrak{h}^{\perp}$ in many cases. Each of these descriptions can be combined with the tables of \`Ela\v{s}vili \cite{MR0304554,MR0304555} to compute $H_{*}$, after which Corollary \ref{cor: gen stab} can be applied. We illustrate this in the following example. 

\begin{example}
Consider the pair $(G,H)=(\SL_{p+q},\mathrm{S}(\GL_{p}\times\GL_{q}))$ with $1\le p\le q$. The vector space $\lah^{\perp}$ is isomorphic to $\left(\bbC^{p}\otimes(\bbC^{q})^{*}\right)\oplus\left((\bbC^{p})^{*}\otimes\bbC^{q}\right)$ as an $H$-representation. The Lie algebra of the generic stabilizer for this action is isomorphic to $\bbC^{p}\oplus\lasl(q-p)$ if $p<q$ and to $\bbC^{p-1}$ if $p=q$. Hence $(G,H)$ is $\laa$-regular if and only if $q-p\le1$. 
\end{example}

We now formulate a numerical criterion for $\mathfrak{a}$-regularity in terms of spherical-geometric invariants.  
Recall that the \textit{rank} $\mathrm{rk}_{G}(Y)$ of a $G$-variety $Y$ is the dimension of $\laa_{Y}$. The \textit{complexity} $c_{G}(Y)$ of $Y$ is the codimension of a generic $B$-orbit in $Y$. We then have the following equalities, which are due to Knop \cite{KnopWeyl}:
\begin{eqnarray}\label{formula invariants}
2c_{G}(G/H)+\mathrm{rk}_{G}(G/H)&=&\dim G-2\dim H+\dim H_{*};\label{formula invariants1}\\
\mathrm{rk}_{G}(G/H)&=&\dim T-\dim T_{*},\label{formula invariants2}
\end{eqnarray}
where $T_{*}$ is a maximal torus of $H_{*}$.

\begin{corollary}\label{cor: num crit}
The pair $(G,H)$ is $\laa$-regular if and only if $c_{G}(G/H)+\mathrm{rk}_{G}(G/H)+\dim H=\dim B$.
\end{corollary}

\begin{proof}
Corollary \ref{cor: gen stab} shows that $(G,H)$ is $\laa$-regular if and only if the identity component in $H_{*}$ is abelian. This is in turn equivalent to $\dim H_{*}=\dim T_{*}$, and the result then follows from \eqref{formula invariants} and \eqref{formula invariants2}.
\end{proof}

The criteria established in Corollaries \ref{cor: gen stab} and \ref{cor: num crit} become effective once we are able to either determine the Cartan space $\laa(\lag,\lah)$ or the generic stabilizer $H_{*}$. The latter is difficult to accomplish in full generality, but Losev's work \cite{MR2362821} makes the former achievable in a systematic way. Losev's method features prominently in the next subsection.  

\subsection{The Cartan space of a homogeneous affine variety}\label{Subsection: The Cartan space of a homogeneous affine space} Continuing with the notation used in \ref{subsection: The cotangent bundle of a homogeneous space}, we recall Losev's algorithm \cite{MR2362821} for determining the Cartan space of $(G,H)$. We begin with the following definition (cf. \cite[Section 10]{TimashevBook}).

\begin{definition}\label{Definition: decomposable}
The pair $(G,H)$ or the corresponding pair $(\lag,\lah)$ is called: 
\begin{itemize}
\item[(i)] \textit{decomposable} if there exist non-zero proper ideals $\lag_{1},\lag_{2}$ in $\lag$ and any ideals $\lah_{1},\lah_{2}$ in $\lah$ such that $\lag=\lag_{1}\oplus\lag_{2}$, $\lah=\lah_{1}\oplus\lah_{2}$, $\lah_{1}\subseteq\lag_{1}$, and $\lah_{2}\subseteq\lag_{2}$;
\item[(ii)] \textit{indecomposable} if it is not decomposable;
\item[(iii)] \textit{strictly indecomposable} if $(\mathfrak{g},[\mathfrak{h},\mathfrak{h}])$ is indecomposable.
\end{itemize}
\end{definition}

We note that the Cartan space of a decomposable pair $(\lag_{1}\oplus\lag_{2},\lah_{1}\oplus\lah_{2})$ is $\laa(\lag_{1},\lah_{1})\oplus \laa(\lag_{2},\lah_{2})$. At the same time, observe that $(x_1,x_2)\in\lag_1\oplus\lag_2$ is a regular element if and only if $x_1$ and $x_2$ are regular elements of $\lag_1$ and $\lag_2$, respectively. These last two sentences imply that $(\lag_{1}\oplus\lag_{2},\lah_{1}\oplus\lah_{2})$ is $\laa$-regular if and only if $(\lag_1,\lah_1)$ and $(\lag_2,\lah_2)$ are $\laa$-regular. Recognizing its relevance to later arguments, we record this conclusion as follows.

\begin{lemma}\label{lem: indecomposable}
Consider a collection of indecomposable pairs $(\mathfrak{g}_i,\mathfrak{h}_i)$, $i=1,\ldots,n$, and suppose that our pair $(\lag,\lah)$ is given by 
\begin{equation}\label{Equation: Decomposition into pairs}(\mathfrak{g},\mathfrak{h})=\bigg(\bigoplus_{i=1}^n\mathfrak{g}_i,\bigoplus_{i=1}^n\mathfrak{h}_i\bigg).\end{equation}
Then $(\mathfrak{g},\mathfrak{h})$ is $\mathfrak{a}$-regular if and only if $(\mathfrak{g}_i,\mathfrak{h}_i)$ is $\mathfrak{a}$-regular for all $i=1,\ldots,n$.
\end{lemma}

\begin{remark}\label{Remark: Decomposition remark}
Note that our pair $(\lag,\lah)$ is necessarily expressible in the form \eqref{Equation: Decomposition into pairs}, i.e. there exist indecomposable pairs $(\lag_i,\lah_i)$, $i\in\{1,\ldots,n\}$, such that $\lag_i$ (resp. $\lah_i$) is an ideal in $\lag$ (resp. $\lah$) for all $i$ and \eqref{Equation: Decomposition into pairs} holds. This observation follows from Definition \ref{Definition: decomposable} via a straightforward induction argument, and it will be used implicitly in some of our arguments.
\end{remark}

We now resume the main discussion. Note that for a subalgebra $\laj\subseteq\lah$, we have an inclusion $\laa(\lag,\lah)\subseteq\laa(\lag,\laj)$ of Cartan spaces. It follows that $\laa(\lag,\lah)\subseteq\laa(\lag,\laj)$ for all ideals $\laj\le\lah$, which leads to the following definition (cf. \cite[Definition 1.1]{MR2362821}).

\begin{definition}
A reductive subalgebra $\laj\subseteq\lag$ is called \textit{essential} if for every proper ideal $\lai\le\laj$, the inclusion $\laa(\lag,\laj)\subseteq\laa(\lag,\lai)$ is strict. 
\end{definition}

Now consider the Lie algebra $\mathfrak{h}_{*}$ of $H_{*}$, where $H_{*}$ is the generic stabilizer for the $H$-action on $\mathfrak{h}^{\perp}$ (see \ref{subsection: The cotangent bundle of a homogeneous space}). Losev shows that $\lah_{*}$ generates an ideal $\lah_{\ess}\le\lah$ that is an essential subalgebra of $\lag$. This essential subalgebra is reductive and has the following properties: 
\begin{itemize}
\item $\lah_{\ess}\le\lah$ is the unique ideal of $\mathfrak{h}$ for which $\laa(\lag,\lah)=\laa(\lag,\lah_{\ess})$;
\item $\lah_{\ess}$ is maximal (for inclusion) among the ideals of $\lah$ that are essential subalgebras of $\lag$.
\end{itemize}
In principle, this reduces the computation of $\laa(\lag,\lah)$ to the task of determining $\lah_{\ess}$ and $\laa(\lag,\lah_{\ess})$.

The preceding discussion allows us to sketch the main results of \cite{MR2362821}. Losev classifies the essential subalgebras $\mathfrak{j}\subseteq\mathfrak{g}$ that are semisimple, and in each such case he presents $\mathfrak{a}(\mathfrak{g},\mathfrak{j})$ as the span of certain linear combinations of fundamental weights. This information may also be used to determine the Cartan space when $\mathfrak{j}$ is non-semisimple, provided that one knows the center of $\mathfrak{j}$. To this end, Losev gives an algorithm for calculating the centers of non-semisimple essential subalgebras.

\subsection{Preliminaries for the classifications}\label{Subsection: preliminaries}
We now discuss four items that are crucial to the classifications in \ref{Subsection: The classification}. Our first item is the following elementary observation. 

\begin{observation}
Let $\mathfrak{r}$ be a complex reductive Lie algebra with a reductive ideal $\lai\le\mathfrak{r}$. If $\laj$ is a reductive ideal in $\lai$, then $\laj$ is also an ideal in $\mathfrak{r}$. This follows immediately from the decomposition of a reductive Lie algebra into a direct sum of its center and simple ideals, and it will be used implicitly in some of what follows.
\end{observation}

We also need the following definition, which serves to formalize a standard idea.

\begin{definition}\label{Definition: Conjugate}
Let $\mathfrak{r}_1$ and $\mathfrak{r}_2$ be complex Lie algebras with respective subalgebras $\mathfrak{s}_1$ and $\mathfrak{s}_2$. We refer to $(\mathfrak{r}_1,\mathfrak{s}_1)$ and $(\mathfrak{r}_2,\mathfrak{s}_2)$ as being \textit{conjugate} if $\mathfrak{r}_1=\mathfrak{r}_2$ and $\mathfrak{s}_1=\phi(\mathfrak{s}_2)$ for some Lie algebra automorphism $\phi:\mathfrak{r}_1\rightarrow\mathfrak{r}_1$. 
\end{definition}

With this in mind, we have the following lemma. 

\begin{lemma}\label{lemma: classification}
Assume that $\mathfrak{g}$ is simple and let $\mathfrak{h}\subseteq\lag$ be a reductive subalgebra.

\begin{itemize}
\item[(i)] If $(\lag,\lai)$ is not conjugate to a pair in Tables 1 or 2 from \cite{MR2362821} for any ideal $\lai\le\lah$, then $\lah_{\ess}=\{0\}$. In this case, $\laa(\lag,\lah)=\mathfrak{t}$ and $(\lag,\lah)$ is $\laa$-regular. 

\item[(ii)] If $\lah_{\ess}\ne\{0\}$, then $(\lag,\lah)$ is $\laa$-regular if and only if $(\lag,[\lah_{\ess},\lah_{\ess}])$ is conjugate to a pair in Table \ref{table: h_ess nontrivial} below. 
\end{itemize}
\end{lemma}

\begin{table}[h!]
\begin{center}
$$ \begin{array}{c|c|c}
& \lag & \lai \\
\hline
1& \lasl_{2k} & \lasl_{k}\oplus\lasl_{k}  \\
2& \lasl_{2k-1} & \lasl_{k}\oplus\lasl_{k-1} \\
3& \lasp_{8}& \lasp_{4}\oplus\lasl_{2}\oplus\lasl_{2}\\
4& \lasp_{6}& \lasl_{2}\oplus\lasl_{2}\oplus\lasl_{2}\\
5&\lae_{6}&\lasl_{6}\\
6& \lasl_{2n+1} & \lasl_{n+1} \\
7& \lasl_{2n+1} & \lasp_{2n} \\
\end{array}$$
  \vspace{0.2cm}
\caption{For each line, the embedding $\lai\subseteq\lag$ is as described in \cite[Section 6]{MR2362821}.}
\label{table: h_ess nontrivial}
\end{center}
\end{table}

\begin{proof}
We begin by proving (i), and thus assume that $(\lag,\lai)$ is not conjugate to a pair in Tables 1 or 2 from \cite{MR2362821} for any ideal $\lai\leq\lah$. Noting the particular classification that each table gives, we conclude that $[\lah_{\ess},\lah_{\ess}]$ cannot contain a non-zero semisimple ideal. Hence $[\lah_{\ess},\lah_{\ess}]=\{0\}$, i.e. $\lah_{\ess}$ is abelian.
Since $\lah_{\ess}$ is also reductive, one can find a Lie algebra automorphism of $\lag$ that sends $\lah_{\ess}$ into $\lat$. This implies that $(\lag,\lah)$ is conjugate to a pair $(\lag,\widetilde{\lah})$ satisfying $\widetilde{\lah}_{\ess}\subseteq\lat$. We may therefore assume that $\lah_{\ess}\subseteq\lat$.
 
Note that the inclusions $\{0\}\subseteq\lah_{\ess}\subseteq\lat$ yield $\laa(\lag,\lat)\subseteq\laa(\lag,\lah_{\ess})\subseteq\laa(\lag,\{0\})$, which by Example \ref{Example: Cartan space of G/T} amounts to the statement $\lat\subseteq\laa(\lag,\lah_{\ess})\subseteq\laa(\lag,\{0\})$. At the same time, the inclusion $\laa(\lag,\{0\})\subseteq\lat$ follows from how we defined Cartan spaces in \ref{subsection: image of moment map}. We conclude that $$\lat=\laa(\lag,\lah_{\ess})=\laa(\lag,\{0\}).$$ Recalling the properties of $\lah_{\ess}$ discussed in \ref{Subsection: The Cartan space of a homogeneous affine space}, the first equality implies that $\lat=\laa(\lag,\lah)$ and the second equality gives $\lah_{\ess}=\{0\}$. The $\laa$-regularity of $(\lag,\lah)$ now follows from the fact that $\lat\cap\lag_{\text{reg}}\neq\emptyset$, completing our proof of (i).

To prove (ii), we first assume that $(\lag,[\lah_{\ess},\lah_{\ess}])$ is conjugate to a pair in Table \ref{table: h_ess nontrivial}. If $\lah_{\ess}$ is semisimple, i.e. $[\lah_{\ess},\lah_{\ess}]=\lah_{\ess}$, then $\laa(\lag,[\lah_{\ess},\lah_{\ess}])=\laa(\lag,\lah_{\ess})=\laa(\lag,\lah)$. This observation and an inspection of \cite[Table 1]{MR2362821} reveal that $(\lag,\lah)$ is $\laa$-regular. If $\lah_{\ess}$ is not semisimple, then $(\lag,[\lah_{\ess},\lah_{\ess}])$ is conjugate to one of items 6 and 7 in Table \ref{table: h_ess nontrivial} and $\lah_{\ess}$ has a non-trivial center $\laz(\lah_{\ess})$. The result \cite[Theorem 1.3(c)]{MR2362821} then shows that $$\laz(\lah_{\ess})\subseteq\laz:=\laz(\laz_{\lag}([\lah_{\ess},\lah_{\ess}])),$$ where $\laz_{\lag}([\lah_{\ess},\lah_{\ess}])$ is the subalgebra of all elements in $\lag$ that commute with every element of $[\lah_{\ess},\lah_{\ess}]$ and $\laz(\laz_{\lag}([\lah_{\ess},\lah_{\ess}]))$ is the center of this subalgebra. Noting again that $(\lag,[\lah_{\ess},\lah_{\ess}])$ is conjugate to item 6 or 7 in Table \ref{table: h_ess nontrivial}, one uses \cite[Table 2]{MR2362821} to see that $\laa(\lag,[\lah_{\ess},\lah_{\ess}]+\laz)$ has regular elements. Note also that $$\lah_{\ess}=[\lah_{\ess},\lah_{\ess}]+\laz(\lah_{\ess})\subseteq[\lah_{\ess},\lah_{\ess}]+\laz$$ implies $$\laa(\lag,[\lah_{\ess},\lah_{\ess}]+\laz)\subseteq\laa(\lag,\lah_{\ess})=\laa(\lag,\lah).$$ The previous two sentences together show that $(\lag,\lah)$ is $\laa$-regular.

For the converse we suppose that $\lah_{\ess}\le\lah$ is not the trivial ideal. The discussion above implies that $\lah_{\ess}$ cannot be abelian, so that $[\lah_{\ess},\lah_{\ess}]\le\lah$ is a semisimple and non-trivial ideal. It then follows from Losev's setup in \cite{MR2362821} that $[\lah_{\ess},\lah_{\ess}]$ is conjugate to a pair in \cite[Table 1]{MR2362821} or \cite[Table 2]{MR2362821}. Hence there are three mutually exclusive possibilities: $(\lag,[\lah_{\ess},\lah_{\ess}])$ is conjugate to a pair in:  
\begin{itemize}
\item[(a)] \cite[Table 1]{MR2362821}, but not to one in \cite[Table 2]{MR2362821};
\item[(b)] \cite[Table 1]{MR2362821} and \cite[Table 2]{MR2362821};
\item[(c)] \cite[Table 2]{MR2362821}, but not to one in \cite[Table 1]{MR2362821}.
\end{itemize}
In each instance, we simply use Losev's tables to inspect all possible Cartan spaces $\laa(\lag,\lah)$ and determine whether each has a regular element. 

We first suppose that (a) holds. Then $(\lag,\lah)$ is $\laa$-regular precisely when $(\lag,\lah_{\ess})$ is conjugate to one of the items 2 (with $k=n/2,(n+1)/2$), 6 (with $n=4$), 7 or 21 from \cite[Table 1]{MR2362821}. These pairs constitute the first five lines of Table \ref{table: h_ess nontrivial}.

Now suppose that (b) holds. Then $(\lag,[\lah_{\ess},\lah_{\ess}])$ is conjugate to one of the items 1, 2 (with $n/2<k\le n-2$), 10 or 19 from \cite[Table 1]{MR2362821}. A case-by-case examination reveals that $(\lag,[\lah_{\ess},\lah_{\ess}])$ is not $\laa$-regular, i.e. $\laa(\lag, [\lah_{\ess},\lah_{\ess}])\cap\lag_{\reg}=\emptyset$. It then follows from the inclusion $\laa(\lag,\lah_{\ess})\subseteq\laa(\lag, [\lah_{\ess},\lah_{\ess}])$ that $\laa(\lag,\lah_{\ess})\cap\lag_{\reg}=\emptyset$. Since $\laa(\lag,\lah_{\ess})=\laa(\lag,\lah)$, this means that $(\lag,\lah)$ is not $\laa$-regular. 

We last suppose that (c) holds, in which case $(\lag,[\lah_{\ess},\lah_{\ess}])$ is conjugate to item 6 or 7 in Table \ref{table: h_ess nontrivial}. As argued above, the pair $(\lag,\lah)$ is necessarily $\laa$-regular.
\end{proof}

For the last preliminary topic, let $H$ be any reductive subgroup of our connected semisimple group $G$. The coordinate ring $\bbC[G/H]$ then decomposes into certain irreducible, highest-weight $G$-modules, and the highest weights appearing in this decomposition are the so-called \textit{spherical weights}. These weights form a finitely generated semigroup $\Lambda_{+}(G,H)$. With this in mind, we record the following immediate consequence of  \cite[Proposition 5.14]{TimashevBook}. 

\begin{lemma}\label{lemma: weight semigroup}
If $H$ is any closed, reductive subgroup of $G$, then $\mathfrak{a}(\mathfrak{g},\mathfrak{h})^*$ is spanned by $\Lambda_{+}(G,H)$.
\end{lemma}

\subsection{The classifications}\label{Subsection: The classification}

We maintain the notation used in \ref{Subsection: The Cartan space of a homogeneous affine space}, and now address the classification of $\mathfrak{a}$-regular pairs $(G,H)$ (equivalently, $\laa$-regular pairs $(\lag,\lah)$) in each of the following three cases: $H$ is a Levi subgroup \ref{sss: Levi}, $H$ is symmetric \ref{sss: sym}, and $H$ is simultaneously reductive, spherical, and non-symmetric \ref{sss: sph}. In each case, we reduce to the classification of strictly indecomposable, $\laa$-regular pairs. We list all conjugacy classes of such pairs in each of the cases \ref{sss: sym} and \ref{sss: sph}, where the notion of conjugacy class comes from Definition \ref{Definition: Conjugate}. 

\begin{remark}\label{Remark: Division into cases}
We emphasize that the classification of strictly indecomposable pairs works differently in each of the above-mentioned cases. In the case of Levi subgroups $H\subseteq G$, the classification is almost entirely based on Losev's work \cite{MR2362821}. This is in contrast to the case of symmetric subgroups, in which we appeal to representation-theoretic results about symmetric spaces. Several of these results are not applicable to the case of a reductive spherical $H\subseteq G$, for which we instead harness the works of Brion \cite{MR906369}, Kr\"amer \cite{MR528837}, and Mikityuk \cite{MR842398}.
\end{remark}

\begin{remark}
Note that every symmetric subgroup of $G$ is reductive and spherical (see \cite[Theorem 26.14]{TimashevBook}). The techniques and arguments in \ref{sss: sph} thereby imply the classification results in \ref{sss: sym}. Despite this, we believe that the representation-theoretic approach taken in \ref{sss: sym} is independently interesting and worthwhile. Further distinctions between \ref{sss: sym} and \ref{sss: sph} are discussed in Remark \ref{Remark: Indecomposable} and Example \ref{Example: Indecomposable}.  
\end{remark}

\subsubsection{Levi subgroups.}\label{sss: Levi} 

Assume that $H$ is a Levi subgroup of $G$, by which we mean that $H$ is a Levi factor of a parabolic subgroup $P\subseteq G$. It follows that $\lah$ is a Levi factor of a parabolic subalgebra $\lap\subseteq\lag$. Now let $\lag=\lag_1\oplus
\cdots\oplus\lag_n$ be the decomposition of $\lag$ into its simple ideals $\lag_1,\ldots,\lag_n$. The parabolic subalgebra $\lap$ is then a sum of parabolic subalgebras $\lap_i\subseteq\lag_i$ for $i=1,\ldots,n$, implying that $\lah$ is a sum of Levi factors $\lah_i\subseteq\lap_i$, $i=1,\ldots,n$. An application of Lemma \ref{lem: indecomposable} then shows that $(\lag,\lah)$ is $\laa$-regular if and only if $(\lag_i,\lah_i)$ is $\laa$-regular for all $i=1,\ldots,n$. It therefore suffices to assume that $\lag$ is simple. Our classification then takes the following form.

\begin{proposition}\label{Proposition: Dynkin} Assume that $\mathfrak{g}$ is simple and that $\mathfrak{h}$ is a Levi subalgebra of $\mathfrak{g}$ with $\lah_{\ess}\neq\{0\}$. The pair $(\mathfrak{g},\mathfrak{h})$ is then $\mathfrak{a}$-regular if and only if it is conjugate to a pair in Table \ref{table: Levi}.
In this table, $\lal_{2}$ is any reductive subalgebra of $\lasl_{2n+1}$ that satisfies the following conditions: $\mathfrak{sl}_{n+1}\cap\lal_2=\{0\}$, $\lal_2$ commutes with $\mathfrak{sl}_{n+1}$, and $\lasl_{n+1}\oplus\lal_{2}$ is a Levi subalgebra of $\lasl_{2n+1}$.\footnote{We are implicitly using the embedding $\mathfrak{sl}_{n+1}\subseteq\mathfrak{sl}_{2n+1}$ from line 6 of Table \ref{table: h_ess nontrivial}.}  
\end{proposition}
\begin{table}[h!]
\begin{center}
$$ \begin{array}{c|c|c}
&\lag & \lal \\
\hline
1&\lasl_{2k}&\mathfrak{s}(\lagl_{k}\oplus\lagl_{k}) \\
2&\lasl_{2k-1} & \mathfrak{s}(\lagl_{k}\oplus\lagl_{k-1}) \\
3&\lae_{6}&\lasl_{6}\oplus\bbC\\
4&\lasl_{2n+1} & \lasl_{n+1}\oplus\lal_{2}\\
\end{array}$$
  \vspace{0.2cm}
\end{center}
\caption{Line 3 is to be understood as follows. Up to Lie algebra automorphism, $\lae_{6}$ contains precisely one subalgebra isomorphic to $\lasl_{6}\oplus\lasl_{2}$ (see \cite[Theorem 5.5, Table 12, and Theorem 11.1]{zbMATH03125755}). By choosing a Cartan subalgebra of $\mathfrak{sl}_2$ and identifying it with $\bbC$, one obtains a unique automorphism class of subalgebras in $\lae_{6}$ that are isomorphic to $\lasl_{6}\oplus\bbC$. This turns out to be a class of Levi subalgebras in $\lae_6$, and the reader may take any of these to be the subalgebra $\lal$ in line 3.}
\label{table: Levi}
\end{table}

\begin{proof}
We first assume that $(\lag,\lah)$ is conjugate to a pair in Table \ref{table: Levi}. A case-by-case analysis reveals that each pair in Table \ref{table: Levi} is $\laa$-regular, implying that $(\lag,\lah)$ is $\laa$-regular.

Conversely, assume that $(\lag,\lah)$ is $\laa$-regular. Lemma \ref{lemma: classification}(ii) then implies the existence of an ideal $\lai$ in $\lah$ for which $(\lag,\lai)$ is conjugate to a pair in Table \ref{table: h_ess nontrivial}. We will therefore begin by finding the pairs in Table \ref{table: h_ess nontrivial} for which this is possible. For each such pair $(\mathfrak{r},\laj)$, we will subsequently find the Levi subalgebras $\lal\subseteq\mathfrak{r}$ that contain $\laj$ as an ideal. Note that $(\lag,\lah)$ will then be conjugate to one of the pairs $(\mathfrak{r},\lal)$ arising in this way. It will then suffice to observe that the aforementioned pairs $(\mathfrak{r},\mathfrak{l})$ appear in Table \ref{table: Levi}.

Let $(\mathfrak{r},\laj)$ be any of the pairs appearing in lines 3,4, and 7 of Table \ref{table: h_ess nontrivial}. Observe that the Dynkin diagram of $\laj$ is not a subdiagram in the Dynkin diagram of $\mathfrak{r}$. At the same time, the Dynkin diagram of any ideal in a Levi subalgebra of $\mathfrak{g}$ must be a subdiagram in the Dynkin diagram of $\mathfrak{g}$. It follows that $(\lag,\lai)$ cannot be conjugate to $(\mathfrak{r},\laj)$ for any ideal $\lai\leq\lah$.

In light of the previous paragraph, we may restrict our attention to the pairs in lines 1,2,5, and 6 of Table \ref{table: h_ess nontrivial}. Let $(\mathfrak{r},\laj)$ be any such pair, recalling that the embedding of $\laj$ into $\mathfrak{r}$ is described in \cite[Section 6]{MR2362821} (cf. the caption of Table \ref{table: h_ess nontrivial}). This description is easily seen to imply that $\laj$ is an ideal in a Levi subalgebra of $\mathfrak{r}$. If $(\mathfrak{r},\mathfrak{j})$ is in one of lines 1,2, and 5 from Table \ref{table: h_ess nontrivial}, then the Dynkin diagram of $\laj$ uniquely determines a Levi subalgebra $\lal\subseteq\mathfrak{r}$ that contains $\laj$ as an ideal. The pair $(\mathfrak{r},\lal)$ is recorded in Table \ref{table: Levi}. If $(\mathfrak{r},\mathfrak{j})$ is in line 6 from Table \ref{table: h_ess nontrivial}, i.e. $\mathfrak{r}=\mathfrak{sl}_{2n+1}$ and $\laj=\mathfrak{sl}_{n+1}$, then there are several Levi subalgebras $\lal\subseteq\mathfrak{r}$ that contain $\laj$ as an ideal. The Dynkin diagram of any such $\lal$ is a subdiagram in the Dynkin diagram of $\mathfrak{sl}_{2n+1}$, and it contains the Dynkin diagram of $\mathfrak{sl}_{n+1}$ as a connected component. It follows that $\lal=\mathfrak{sl}_{n+1}\oplus\lal_2$ for some reductive subalgebra $\lal_2\subseteq\mathfrak{sl}_{2n+1}$ that satisfies the desired hypotheses.
\end{proof}

\subsubsection{Symmetric subgroups}\label{sss: sym}

Using the notation established in \ref{subsection: image of moment map} and the introduction of Section \ref{Section: The hyperkahler geometry of T*(G/H)}, we assume that the subgroup $H\subseteq G$ is symmetric. This means that $H$ is an open subgroup of $G^{\theta}$, the subgroup of fixed points of an involutive algebraic group automorphism $\theta:G\to G$. It follows that $(\lag,\lah)$ is a symmetric pair, i.e. $\lah$ coincides with the set of $\theta$-fixed vectors $\lag^{\theta}\subseteq\lag$ for the corresponding involutive Lie algebra automorphism $\theta:\lag\rightarrow\lag$.

\begin{lemma}\label{Lemma: Symmetric}
If $\lah$ is any reductive subalgebra of $\lag$, then $(\lag,\lah)$ is a symmetric pair if and only if there exist strictly indecomposable symmetric pairs $(\lag_i,\lah_i)$, $i\in\{1,\ldots,n\}$, such that $\lag_i$ (resp. $\lah_i$) is an ideal in $\lag$ (resp. $\lah$) for all $i$ and $$(\mathfrak{g},\mathfrak{h})=\bigg(\bigoplus_{i=1}^n\mathfrak{g}_i,\bigoplus_{i=1}^n\mathfrak{h}_i\bigg).$$ 
\end{lemma}

\begin{proof}
The backward implication follows from the following simple observation: if $(\lag_1,\lah_1)$ and $(\lag_2,\lah_2)$ are symmetric pairs, then $(\lag_1\oplus\lag_2,\lah_1\oplus\lah_2)$ is also a symmetric pair.

To prove the forward implication, assume that $(\lag,\lah)$ is a symmetric pair and let $\theta:\lag\rightarrow\lag$ be an involutive automorphism for which $\lah=\lag^{\theta}$. Note that each simple ideal of $\lag$ is either $\theta$-stable or interchanged by $\theta$ with a different simple ideal. We may therefore identify $\lag$ with
$$\lag_1^{\oplus 2}\cdots\oplus\lag_s^{\oplus 2}\oplus\lag_{s+1}\cdots\oplus\lag_{s+t}$$ for simple Lie algebras $\lag_1,\ldots,\lag_{s+t}$, such that $\theta$ becomes the following map: $(x,y)\mapsto (y,x)$ on each summand $\lag_i^{\oplus 2}$ and $x\mapsto\theta_j(x)$ on each summand $\mathfrak{g}_j$, where $\theta_j:\lag_j\mapsto\lag_j$ is an involutive automorphism. It follows that
$$\lah=\lag^{\theta}=\mathrm{diag}(\lag_1)\oplus\cdots\oplus\mathrm{diag}(\lag_s)\oplus\lag_{s+1}^{\theta_{s+1}}\oplus\cdots\oplus\lag_{s+t}^{\theta_{s+t}},$$ where $\mathrm{diag}(\mathfrak{g}_i):=\{(x,y)\in\mathfrak{g}_i^{\oplus 2}:x=y\}$ for all $i\in\{1,\ldots,s\}$.

In light of the above, it suffices to prove that the symmetric pairs $(\lag_i^{\oplus 2},\mathrm{diag}(\lag_i))$ and $(\lag_j,\lag_j^{\theta_j})$ are strictly indecomposable for all $i\in\{1,\ldots,s\}$ and $j\in\{s+1,\ldots,s+t\}$. The strict indecomposability of the latter pair follows from the fact that $\lag_j$ is simple. Now observe that the simplicity of $\lag_i\cong\mathrm{diag}(\lag_i)$ implies that $(\lag_i^{\oplus 2},[\mathrm{diag}(\lag_i),\mathrm{diag}(\lag_i)])=(\lag_i^{\oplus 2},\mathrm{diag}(\lag_i))$. It follows that $(\lag_i^{\oplus 2},\mathrm{diag}(\lag_i))$ is strictly indecomposable if and only if it is indecomposable. However, since $\mathrm{diag}(\lag_i)$ is simple, the decomposability of $(\lag_i^{\oplus 2},\mathrm{diag}(\lag_i))$ would entail $\mathrm{diag}(\lag_i)$ being contained in a proper ideal of $\lag_i^{\oplus 2}$. This is not possible, meaning that $(\lag_i^{\oplus 2},\mathrm{diag}(\lag_i))$ is indeed strictly indecomposable. The proof is complete.
\end{proof}

\begin{remark}\label{Remark: Indecomposable}
One immediate consequence is that every indecomposable symmetric pair $(\lag,\lah)$ is strictly indecomposable. This is not true of an arbitrary reductive spherical pair $(\lag,\lah)$ (see Example \ref{Example: Indecomposable}).
\end{remark}

Together with Lemma \ref{lem: indecomposable}, Lemma \ref{Lemma: Symmetric} reduces the classification of $\laa$-regular symmetric pairs to the classification of $\laa$-regular, strictly indecomposable symmetric pairs. Let $(\lag,\lah)$ be a pair of the latter sort, and let $(G,H)$ denote an associated pair of groups. Let us also consider an involutive automorphism $\theta:\lag\rightarrow\lag$ satisfying $\lah=\lag^{\theta}$. This forms part of the eigenspace decomposition $\lag=\lah\oplus\laq$, where $\mathfrak{q}\subseteq\mathfrak{g}$ is the $-1$-eigenspace of $\theta$. One can then find a maximal abelian subspace $\mathfrak{c}\subseteq\laq$, meaning that $\mathfrak{c}$ is a vector subspace of $\mathfrak{q}$ that is maximal with respect to the following condition: $\mathfrak{c}$ is abelian and consists of semisimple elements in $\lag$ (cf. \cite[Corollary 37.5.4]{Tauvel}).

Now recall our discussion of the generic stabilizer $H_{*}\subseteq H$ and its Lie algebra $\lah_{*}\subseteq\lah$ (see \ref{subsection: The cotangent bundle of a homogeneous space} and \ref{Subsection: The Cartan space of a homogeneous affine space}).  At the same time, let $\laz_{\lah}(Y)$ denote the subalgebra of all $x\in\lah$ that commute with every vector in a subset $Y\subseteq\lag$.

\begin{lemma}\label{lemma: m = h_*}
We have $\lah_{*}=\laz_{\lah}(\lac)$.
\end{lemma}

\begin{proof}
The $H$-module isomorphisms $\mathfrak{h}^{\perp}\cong\mathfrak{g}/\mathfrak{h}\cong\mathfrak{q}$ imply that $H_{*}$ is a generic stabilizer for the $H$-action on $\laq$. Note also that $H\cdot\lac\subseteq\laq$ is dense (see \cite[Lemma 38.7.1]{Tauvel}) and constructible. It follows that $\lah_{*}=\laz_{\lah}(c)$ for all $c$ in an open dense subset $\lac_{1}\subseteq\lac$.
At the same time, there exists an open dense subset $\lac_{2}\subseteq\lac$ with the property that $\laz_{\lah}(c)=\laz_{\lah}(\lac)$ for all $c\in\lac_{2}$ (see the proof of \cite[Proposition 38.4.5]{Tauvel}). 
Hence $\lah_{*}=\laz_{\lah}(\lac)$, as desired.
\end{proof}

\begin{remark}With Remark \ref{Remark: Generic stabilizer} in mind, one can phrase Lemma \ref{lemma: m = h_*} as follows: $\laz_{\lah}(\lac)$ represents the conjugacy class of Lie algebras of generic stabilizers for the $H$-action on $\lah^{\perp}$.\end{remark}

We now explain the classification of $\laa$-regular, strictly indecomposable symmetric pairs $(\lag,\lah)$.  Up to conjugation (see Definition \ref{Definition: Conjugate}), such pairs are parametrized by Satake diagrams (see \cite[Section 26.5]{TimashevBook}). The Satake diagram for a symmetric pair $(\lag,\lah)$ is the Dynkin diagram of $\lag$, together with extra decorations that encode the associated involution $\theta:\lag\to\lag$.
Part of this decoration consists of painting some of the nodes black; these are precisely the simple roots of $\laz_{\lah}(\lac)$.
At the same time, recall that Lemma \ref{lemma: m = h_*} identifies $\laz_{\lah}(\lac)$ with $\lah_{*}$. Appealing to Corollary \ref{cor: gen stab}, we see that the $\laa$-regularity of $(\lag,\lah)$ is equivalent to the Satake diagram of $(\lag,\lah)$ having none of its nodes painted black. This leads to the following result.

\begin{proposition}\label{Proposition: Symmetric proposition}
A strictly indecomposable symmetric pair $(\lag,\lah)$ is $\laa$-regular if and only if it is conjugate to one of the pairs in Table \ref{table: symmetric}. In this table, $\las$ denotes any simple Lie algebra.
\end{proposition}

\begin{table}[h!]
\begin{center}
$$ \begin{array}{c|c|c}
 & \lag & \lah  \\
\hline
1 & \lasl_{n}& \laso_{n} \\
\hline
2 & \lasl_{2n+1} & \lasl_{n+1}\oplus\lasl_{n}\oplus\bbC\\
 & \lasl_{2n} & \lasl_{n}\oplus\lasl_{n}\oplus\bbC \\
\hline
3 & \laso_{2n+1} & \laso_{n+1}\oplus\laso_{n}\\
 & \laso_{2n} & \laso_{n}\oplus\laso_{n} \\
 & \laso_{2n} & \laso_{n-1}\oplus\laso_{n+1} \\
\hline
4 & \lasp_{2n} & \lagl_{n} \\
\hline
5 & \lae_{6} & \lasp_{8} \\
\hline
6 & \lae_{6} & \lasl_{6}\oplus\lasl_{2}\\
\hline
7 & \lae_{7} & \lasl_{8} \\
\hline
8 & \lae_{8} & \laso_{16} \\
\hline
9 & \laf_{4} & \lasp_{6}\oplus \lasl_{2} \\
\hline
10 & \lag_{2} & \lasl_{2}\oplus\lasl_{2} \\
\hline
11 & \las\oplus\las & \diag(\las) 
\end{array}$$
\vspace{0.2cm}
\caption{The embeddings $\lah\subseteq\lag$ are obtained from \cite[Table 1]{MR528837}, which describes each embedding on the level of algebraic groups.}
\label{table: symmetric}
\end{center}
\end{table}

\begin{proof}
Following the discussion above, we only need to list the symmetric pairs whose Satake diagrams have no black nodes. These diagrams can be found in \cite[Table 26.3]{TimashevBook}, and the result follows from an inspection of this table.
\end{proof}

\subsubsection{Reductive spherical subgroups}\label{sss: sph}
Using the notation in \ref{subsection: image of moment map} and the introduction of Section \ref{Section: The hyperkahler geometry of T*(G/H)}, we additionally assume that $(G,H)$ and $(\lag,\lah)$ are \textit{reductive spherical pairs}. This means that $H$ is a reductive spherical subgroup of $G$, i.e. $H$ is reductive and $B$ has an open orbit in $G/H$. Note that this is equivalent to $\lah$ being a reductive subalgebra of $\lag$ satisfying $\widetilde{\lab}+\lah=\lag$ for some Borel subalgebra $\widetilde{\lab}\subseteq\lag$ (see \cite[Section 25.1]{TimashevBook}).
We shall sometimes also require $(G,H)$ and $(\lag,\lah)$ to be non-symmetric, noting that the classification in \ref{sss: sym} renders this a harmless assumption.     

\begin{example}\label{Example: Indecomposable}
In contrast to the situation considered in Remark \ref{Remark: Indecomposable}, an indecomposable reductive spherical pair need not be strictly indecomposable. Set $\lag=\lasl_{n+1}\oplus\lasl_{2}$ and let $\lah\subseteq\lag$ be the image of
$$\lasl_{n}\oplus\bbC\to \lag,\quad (A,t)\mapsto (\diag(A+tI_{n},-nt),\diag(t,-t)),\quad (A,t)\in \lasl_{n}\oplus\bbC.$$
This is an indecomposable spherical pair, but it is not strictly indecomposable.
\end{example}

\begin{remark}
The strictly indecomposable reductive spherical pairs $(G,H)$ have been classified by Kr\"amer \cite{MR528837} for $G$ simple, and by Mikityuk \cite{MR842398} and Brion \cite{MR906369} for $G$ semisimple.
\end{remark}

We begin by assuming that our reductive spherical pair $(G,H)$ is strictly indecomposable. Now note that Lemma \ref{lemma: weight semigroup} allows us to investigate $\mathfrak{a}$-regularity via $\Lambda_{+}(G,H)$, and the case-by-case analyses of \cite{MR528837} thereby become important. The aforementioned reference gives explicit semigroup generators of $\Lambda_{+}(G,H)$ if $G$ is simple. If $G$ is only semisimple, then a description of $\Lambda_{+}(G,H)$ can be obtained from \cite[Table 1]{MR2779106} as follows. If $\lah$ has a trivial center, then generators of $\Lambda_{+}(G,H)$ are given in \cite[Table 1]{MR2779106}. If $\lah$ has a non-trivial center, then \cite[Table 1]{MR2779106} provides a finite set $\{(\lambda_{1},\chi_{1}),\ldots,(\lambda_{s},\chi_{s})\}$ of generators for the so-called \textit{extended weight semigroup} of $(G,H)$. The $\lambda_{i}$ are dominant weights for $G$ and the $\chi_{i}$ are characters of $H$. The weight semigroup $\Lambda_{+}(G,H)$ identifies with the collection of all points in the extended weight semigroup that have the form $(\lambda,0)$. This amounts to the following statement: a dominant weight $\lambda$ belongs to $\Lambda_{+}(G,H)$ if and only if $\lambda=\sum_{i=1}^{s}n_{i}\lambda_{i}$ for some non-negative integers $n_{i}$ satisfying $\sum_{i=1}^{s}n_{i}\chi_{i}=0$.
Together with an inspection of \cite[Table 1]{MR528837} and \cite[Table 1]{MR2779106}, this discussion yields the following fact.

\begin{lemma}\label{Lemma: Spherical lemma}
If $(G,H)$ is a strictly indecomposable reductive spherical pair that is not symmetric, then $(G,H)$ is $\laa$-regular if and only if $(\lag,\lah)$ is conjugate to a pair in Table \ref{table spherical}.
\end{lemma}  

\begin{table}[h!]
\begin{center}
$$ \begin{array}{c|c|c}
 & \lag & \lah  \\
\hline
1 & \lasl_{2n+1} & \lasl_{n+1}\oplus\lasl_{n} \\
\hline
2 & \lasl_{2n+1} & \lasp_{2n}\oplus\bbC\\
\hline
3 & \lasl_{2n+1} & \lasp_{2n}\\
\hline
4 & \laso_{2n+1} & \lagl_{n} \\
\hline
5 & \lasl_{n+1}\oplus\lasl_{n} & \lasl_{n}\oplus\bbC\\
\hline
6 & \laso_{n+1}\oplus\laso_{n} & \laso_{n}\\
\hline
7& \lasl_{n}\oplus\lasp_{2m} & \lagl_{n-2}\oplus\lasl_{2}\oplus\lasp_{2m-2}\\
&  (n=3,4,5, m=1,2) &\\
\hline
8 & \lasl_{n}\oplus\lasp_{2m} & \lasl_{n-2}\oplus\lasl_{2}\oplus\lasp_{2m-2}\\
 &  (n=3,5, m=1,2) \\
\hline
9 &\lasp_{2m}\oplus\lasp_{2n} & \lasp_{2n-2}\oplus\lasp_{2}\oplus\lasp_{2m-2}\\
& (m,n=1,2) \\
\hline
10 & \lasp_{2n}\oplus\lasp_{4} & \lasp_{2n-4}\oplus\lasp_{4}\\
& (n=3,4) &\\
\hline
11 & \lasp_{2\ell}\oplus\lasp_{2m}\oplus\lasp_{2n} & \lasl_{2\ell-2}\oplus\lasp_{2m-2}\oplus\lasp_{2n-2}\oplus\lasp_{2} \\
 &(\ell,m,n=1,2) &\\
\hline
12 & \lasp_{2n}\oplus\lasp_{4}\oplus\lasp_{2m} & \lasp_{2n-2}\oplus\lasp_{2}\oplus\lasp_{2}\oplus\lasp_{2n-2} \\
& (n,m=1,2) &\\
\end{array}$$
  \vspace{0.2cm}
\caption{The embeddings $\lah\subseteq\lag$ are obtained from \cite[Table 1]{MR528837} and \cite[Theorem 0]{MR906369}, which describe each embedding on the level of algebraic groups.
}\label{table spherical}
\end{center}
\end{table}

Together with Proposition \ref{Proposition: Symmetric proposition}, this result classifies the $\laa$-regular, strictly indecomposable reductive spherical pairs. A natural next step is to study the indecomposable reductive spherical pairs that are $\laa$-regular, for which we need the following lemma. 

\begin{lemma}\label{lem: inherited regularity}
Let $(\lag,\lah)$ be a strictly indecomposable reductive spherical pair. If $(\lag,[\lah,\lah])$ is $\laa$-regular, then $(\lag,
\lah)$ is $\laa$-regular.
\end{lemma}

\begin{proof}
The statement is obviously true if $\lah$ is semisimple, so we assume $\lah$ to be non-semisimple. Let us prove the statement by contraposition, assuming that $(\lag,\lah)$ is a strictly indecomposable reductive spherical pair that is not $\laa$-regular. At the same time, $\lah$ being non-semisimple and an inspection of \cite[Table 26.3]{TimashevBook}, \cite[Table 1]{MR528837}, and \cite[Table 1]{MR2779106} reveal that $(\lag,\lah)$ is conjugate to one of the pairs in Table \ref{table: not regular} below. It therefore suffices to prove the following claim: if $(\lag,\lah)$ is conjugate to a pair in Table \ref{table: not regular}, then $(\lag,[\lah,\lah])$ is not $\laa$-regular.

\begin{table}[h!]
\begin{center}
$$ \begin{array}{c|c|c}
&\lag & \lah \\
\hline
1& \lasl_{p+q}\quad (|p-q|>1) & \lasl_{p}\oplus\lasl_{q}\oplus\bbC\\
\hline
2&\laso_{2n}&\lagl_{n}\\
\hline
3&\lae_{6}&\laso(10)\oplus\bbC\\
\hline
4&\lae_{7}&\lae_{6}\oplus\bbC\\
\hline
5&\lasp_{2n}\quad(n>2)&\lasp_{2n-2}\oplus\bbC\\
\hline
6&\laso_{10}&\laso_{7}\oplus\laso_{2}\\
\hline
7&\lasl_{n}\oplus\lasp_{2m}\quad(\mbox{$n>6$ or $m>2$})&\lagl_{n-2}\oplus\lasl_{2}\oplus\lasp_{2m-2}
\end{array}$$
  \vspace{0.2cm}
\end{center}
\caption{The embeddings $\lah\subseteq\lag$ are as described in \cite[Table 1]{MR528837} and \cite[Theorem 0]{MR906369}, where they are given as embeddings of the corresponding algebraic groups.}
\label{table: not regular}
\end{table}

Suppose that $(\lag,\lah)$ is conjugate to a pair in lines 1,2,3, or 7 of Table \ref{table: not regular}. It then follows that $(\lag,[\lah,\lah])$ is a strictly indecomposable reductive spherical pair, as it appears in at least one of the classifications of 
Kr\"amer \cite{MR528837}, Mikityuk \cite{MR842398}, and Brion \cite{MR906369}. At the same time, one can verify that $(\lag,[\lah,\lah])$ is not conjugate to a pair in Table \ref{table: symmetric} or Table \ref{table spherical}. Proposition \ref{Proposition: Symmetric proposition} and Lemma \ref{Lemma: Spherical lemma} then imply that $(\lag,[\lah,\lah])$ is not $\laa$-regular. 

Now assume that $(\lag,\lah)$ is conjugate to one of the remaining pairs in Table \ref{table: not regular}. Let $(G,H)$ be a corresponding reductive spherical pair of groups, and let us take $G$ to be simply-connected. We note that \cite[Table 10.2]{TimashevBook} then provides explicit generators of $\Lambda_{+}(G,[H,H])$. It is now straightforward to apply Lemma \ref{lemma: weight semigroup} and conclude that $(\lag,[\lah,\lah])$ is not $\laa$-regular.
\end{proof}

We now study the $\laa$-regular, indecomposable reductive spherical pairs. Let $(\lag,\lah)$ be an indecomposable reductive spherical pair and note that $(\lag,[\lah,\lah])$ has the following form (cf. Remark \ref{Remark: Decomposition remark}):
$$(\lag,[\lah,\lah])=\bigg(\bigoplus_{i=1}^n\lag_i,\bigoplus_{i=1}^n\widetilde{\lah}_i\bigg),$$
where for all $i\in\{1,\ldots,n\}$, $\widetilde{\lah}_i$ is a semisimple ideal in $[\lah,\lah]$, $\lag_i$ is a reductive ideal in $\lag$ containing $\widetilde{\lah}_i$, and $(\lag_i,\widetilde{\lah}_{i})$ is indecomposable. Note that each pair $(\lag_i,\widetilde{\lah}_{i})$ is actually strictly indecomposable, owing to the fact that $\widetilde{\lah}_i$ is semisimple. 

Let $\pi_{i}:\lag\to\lag_{i}$ denote the projection onto the $i^{\text{th}}$ factor and set $\laz_{i}:=\pi_{i}(\laz(\lah))$, where $\laz(\lah)$ is the center of $\lah$. It is clear that $\laz_{i}$ is reductive and that it commutes with $\widetilde{\lah}_{i}$, from which we conclude that $\overline{\lah}_{i}:=\widetilde{\lah}_{i}\oplus\laz_{i}\subseteq\lag_{i}$ is a reductive subalgebra. Now set
$$\overline{\lah}:=\bigoplus_{i=1}^n\overline{\lah}_i\subseteq\lag.$$
It follows by construction that $[\lah,\lah]\subseteq\overline{\lah}$ and $\laz(\lah)\subseteq\bigoplus_{i=1}^{n}\laz_{i}\subseteq\overline{\lah}$, implying that $\lah\subseteq\overline{\lah}$ and $\widetilde{\lab}+\lah\subseteq\widetilde{\lab}+\overline{\lah}$ for any Borel subalgebra $\widetilde{\lab}\subseteq\lag$. Since $(\lag,\lah)$ is a reductive spherical pair, the previous sentence shows $(\lag,\overline{\lah})$ to be a reductive spherical pair. Our next result establishes that $(\lag_{i},\overline{\lah}_{i})$ is a reductive spherical pair for all $i\in\{1,\ldots,n\}$.

\begin{lemma}\label{Lemma: Reduction to indecomposable}
Let $(\lag,\lah)$ be an indecomposable reductive spherical pair and use the notation from above. Then $(\lag_i,\overline{\lah}_i)$ is a strictly indecomposable reductive spherical pair for all $i\in\{1,\ldots,n\}$.  
\end{lemma}

\begin{proof}
Since $(\lag,\lah)$ is spherical, there exists a Borel subalgebra $\widetilde{\lab}\subseteq\lag$ satisfying $\widetilde{\lab}+\lah=\lag$. The decomposition $\lag=\lag_{1}\oplus\cdots\oplus\lag_{n}$ gives rise to a decomposition of the form $\widetilde{\lab}=\lab_{1}\oplus\cdots\oplus\lab_{n}$, where $\lab_{i}$ is a Borel subalgebra of $\lag_i$ for all $i\in\{1,\ldots,n\}$. Now note that $\lab_{i}+\overline{\lah}_{i}=\lag_{i}$ for all $i\in\{1,\ldots,n\}$ if and only if $\widetilde{\lab}+\overline{\lah}=\lag$. Recalling that $(\lag,\overline{\lah})$ is a reductive spherical pair, the previous sentence implies that $(\lag_i,\overline{\lah}_i)$ is a reductive spherical pair for all $i\in\{1,\ldots,n\}$.

To complete the proof, we observe that $[\overline{\lah}_{i},\overline{\lah}_{i}]=\widetilde{\lah}_{i}$ for all $i\in\{1,\ldots,n\}$. The strict indecomposability of $(\lag_i,\overline{\lah}_i)$ thus follows from the indecomposability of $(\lag,\widetilde{\lah}_{i})$.
\end{proof}

We may now relate the $\laa$-regularity of $(\lag,\lah)$ to that of $(\lag,\overline{\lah})$.

\begin{proposition}\label{prop class}
Let $(\lag,\lah)$ be an indecomposable reductive spherical pair and use the notation from above. Then $(\lag,\lah)$ is $\laa$-regular if and only if $(\lag,\overline{\lah})$ is $\laa$-regular.
\end{proposition}

\begin{proof}
The inclusion of subalgebras $[\lah,\lah]\subseteq\lah\subseteq\overline{\lah}$ implies the inclusion of Cartan spaces $\laa(\lag,\overline{\lah})\subseteq\laa(\lag,\lah)\subseteq\laa(\lag,[\lah,\lah])$, from which we deduce the backward implication.

For the forward implication, suppose that $(\lag,\lah)$ is $\laa$-regular. The inclusion $\laa(\lag,\lah)\subseteq\laa(\lag,[\lah,\lah])$ then shows $(\lag,[\lah,\lah])$ to be $\laa$-regular, which is equivalent to all of the strictly indecomposable pairs $(\lag_{i},\widetilde{\lah}_{i})$ being $\laa$-regular (see Lemma \ref{lem: indecomposable}). Since $(\lag_{i},\overline{\lah}_{i})$ is a strictly indecomposable reductive spherical pair (see Lemma \ref{Lemma: Reduction to indecomposable}) with $[\overline{\lah}_{i},\overline{\lah}_{i}]=\widetilde{\lah}_{i}$, Lemma \ref{lem: inherited regularity} implies that $(\lag_{i},\overline{\lah}_{i})$ must be $\laa$-regular. It then follows from Lemma \ref{lem: indecomposable} that $(\lag,\overline{\lah})$ is $\laa$-regular.
\end{proof}

We now connect this discussion of $\laa$-regularity for indecomposable reductive spherical pairs to the overarching objective --- a classification of $\laa$-regular reductive spherical pairs. The following lemma is a crucial step in this direction.

\begin{lemma}\label{Lemma: New reduction to indecomposable}
If $\lah$ is any reductive subalgebra of $\lag$, then $(\lag,\lah)$ is a reductive spherical pair if and only if there exist indecomposable reductive spherical pairs $(\lag_i,\lah_i)$, $i\in\{1,\ldots,n\}$, such that $\lag_i$ (resp. $\lah_i$) is an ideal in $\lag$ (resp. $\lah$) for all $i$ and 
$$(\mathfrak{g},\mathfrak{h})=\bigg(\bigoplus_{i=1}^n\mathfrak{g}_i,\bigoplus_{i=1}^n\mathfrak{h}_i\bigg).$$
\end{lemma}

\begin{proof}
By virtue of Remark \ref{Remark: Decomposition remark}, one can find indecomposable pairs $(\lag_i,\lah_i)$ satisfying the above-advertised properties. The proof then becomes entirely analogous to that of Lemma \ref{Lemma: Reduction to indecomposable}. 
\end{proof}

The classification of $\laa$-regular reductive spherical pairs is now described as follows. By virtue of Lemmas \ref{lem: indecomposable} and \ref{Lemma: New reduction to indecomposable}, it suffices to classify the indecomposable reductive spherical pairs that are $\laa$-regular. We thus suppose that $(\lag,\lah)$ is any indecomposable reductive spherical pair. If $(\lag,\lah)$ is strictly indecomposable, then it is $\laa$-regular if and only if it is conjugate to a pair in Table \ref{table: symmetric} or Table \ref{table spherical}. If $(\lag,\lah)$ is not strictly indecomposable, then we consider the associated pair $(\lag,\overline{\lah})$. The $\laa$-regularity of $(\lag,\lah)$ is then equivalent to that of $(\lag,\overline{\lah})$ (see Proposition \ref{prop class}). This is in turn equivalent to every strictly indecomposable pair $(\lag_{i},\overline{\lah}_{i})$ being $\laa$-regular (see Lemma \ref{lem: indecomposable}), which can be assessed via Tables \ref{table: symmetric} and \ref{table spherical}.

\begin{remark}
One might ask about the feasibility of classifying the $\laa$-regular reductive spherical pairs $(G,H)$ satisfying $c_{G}(G/H)>0$. The complexity-one case might be tractable, largely because the papers \cite{MR2089817} and \cite{MR1708594} classify all strictly indecomposable reductive spherical pairs $(G,H)$ with $c_{G}(G/H)=1$. One can thereby determine which of the strictly indecomposable, complexity-one pairs are $\laa$-regular. In analogy with \ref{sss: sym} and \ref{sss: sph}, this might imply a classification of all reductive spherical $(G,H)$ with $c_{G}(G/H)=1$. The case of $c_{G}(G/H)>1$ remains unclear to us. 
\end{remark}

\bibliographystyle{acm} 
\bibliography{Spherical}

\begin{thebibliography}{10}

\bibitem{AbeCrooks}
{\sc Abe, H., and Crooks, P.}
\newblock Hessenberg varieties, {S}lodowy slices, and integrable systems.
  arxiv:1807.07792 (2018), 36pp. {T}o appear in \textit{{M}ath. {Z}}.

\bibitem{MR2089817}
{\sc Arzhantsev, I.~V., and Chuvashova, O.~V.}
\newblock Classification of affine homogeneous spaces of complexity one.
\newblock {\em Mat. Sb. 195}, 6 (2004), 3--20.

\bibitem{MR2779106}
{\sc Avdeev, R.~S.}
\newblock Extended weight semigroups of affine spherical homogeneous spaces of
  nonsimple semisimple algebraic groups.
\newblock {\em Izv. Ross. Akad. Nauk Ser. Mat. 74}, 6 (2010), 3--26.

\bibitem{Bielawski}
{\sc Bielawski, R.}
\newblock Hyperk\"ahler structures and group actions.
\newblock {\em J. London Math. Soc. (2) 55}, 2 (1997), 400--414.

\bibitem{BielawskiComplex}
{\sc Bielawski, R.}
\newblock Slices to sums of adjoint orbits, the {A}tiyah-{H}itchin manifold,
  and {H}ilbert schemes of points.
\newblock {\em Complex Manifolds 4\/} (2017), 16--36.

\bibitem{Biquard}
{\sc Biquard, O.}
\newblock Sur les \'equations de {N}\"ahm et la structure de {P}oisson des
  alg\`ebres de {L}ie semi-simples complexes.
\newblock {\em Math. Ann. 304}, 2 (1996), 253--276.

\bibitem{MR906369}
{\sc Brion, M.}
\newblock Classification des espaces homog\`enes sph\'{e}riques.
\newblock {\em Compositio Math. 63}, 2 (1987), 189--208.

\bibitem{CrooksBulletin}
{\sc Crooks, P.}
\newblock An equivariant description of certain holomorphic symplectic
  varieties.
\newblock {\em Bull. Aust. Math. Soc. 97}, 2 (2018), 207--214.

\bibitem{CrooksRayan}
{\sc Crooks, P., and Rayan, S.}
\newblock Abstract integrable systems on hyperk\"ahler manifolds arising from
  {S}lodowy slices. arxiv:1706.05819 (2017), 17pp. {T}o appear in
  \textit{{M}ath. {R}es. {L}ett}.

\bibitem{Dancer}
{\sc Dancer, A., and Swann, A.}
\newblock Hyperk\"ahler metrics associated to compact {L}ie groups.
\newblock {\em Math. Proc. Cambridge Philos. Soc. 120}, 1 (1996), 61--69.

\bibitem{zbMATH03125755}
{\sc {Dynkin}, E.~B.}
\newblock {Semisimple subalgebras of semisimple Lie algebras.}
\newblock {\em {Transl., Ser. 2, Am. Math. Soc.} 6\/} (1957), 111--243.

\bibitem{MR0304554}
{\sc \`Ela\v{s}vili, A.~G.}
\newblock Canonical form and stationary subalgebras of points in general
  position for simple linear {L}ie groups.
\newblock {\em Funkcional. Anal. i Prilo\v{z}en. 6}, 1 (1972), 51--62.

\bibitem{MR0304555}
{\sc \`Ela\v{s}vili, A.~G.}
\newblock Stationary subalgebras of points of general position for irreducible
  linear {L}ie groups.
\newblock {\em Funkcional. Anal. i Prilo\v{z}en. 6}, 2 (1972), 65--78.

\bibitem{HitchinSelf}
{\sc Hitchin, N.~J.}
\newblock The self-duality equations on a {R}iemann surface.
\newblock {\em Proc. London Math. Soc. (3) 55}, 1 (1987), 59--126.

\bibitem{Hitchin}
{\sc Hitchin, N.~J., Karlhede, A., Lindstr\"om, U., and Ro\v{c}ek, M.}
\newblock Hyperk\"ahler metrics and supersymmetry.
\newblock {\em Comm. Math. Phys. 108}, 4 (1987), 535--589.

\bibitem{KnopWeyl}
{\sc Knop, F.}
\newblock Weylgruppe und {M}omentabbildung.
\newblock {\em Invent. Math. 99}, 1 (1990), 1--23.

\bibitem{KnopAsymptotic}
{\sc Knop, F.}
\newblock The asymptotic behavior of invariant collective motion.
\newblock {\em Invent. Math. 116}, 1-3 (1994), 309--328.

\bibitem{Kostant}
{\sc Kostant, B.}
\newblock Lie group representations on polynomial rings.
\newblock {\em Amer. J. Math. 85\/} (1963), 327--404.

\bibitem{Kovalev}
{\sc Kovalev, A.~G.}
\newblock N\"ahm's equations and complex adjoint orbits.
\newblock {\em Quart. J. Math. Oxford Ser. (2) 47}, 185 (1996), 41--58.

\bibitem{MR0376965}
{\sc Kr\"{a}mer, M.}
\newblock Eine {K}lassifikation bestimmter {U}ntergruppen kompakter
  zusammenh\"{a}ngender {L}iegruppen.
\newblock {\em Comm. Algebra 3}, 8 (1975), 691--737.

\bibitem{MR528837}
{\sc Kr\"{a}mer, M.}
\newblock Sph\"{a}rische {U}ntergruppen in kompakten zusammenh\"{a}ngenden
  {L}iegruppen.
\newblock {\em Compositio Math. 38}, 2 (1979), 129--153.

\bibitem{Kronheimer}
{\sc Kronheimer, P.}
\newblock A hyperk\"ahler structure on the cotangent bundle of a complex {L}ie
  group. ar{X}iv:math/0409253 (2004), 11 pp.

\bibitem{KronheimerSemisimple}
{\sc Kronheimer, P.~B.}
\newblock A hyperk\"ahlerian structure on coadjoint orbits of a semisimple
  complex group.
\newblock {\em J. London Math. Soc. (2) 42}, 2 (1990), 193--208.

\bibitem{KronheimerNilpotent}
{\sc Kronheimer, P.~B.}
\newblock Instantons and the geometry of the nilpotent variety.
\newblock {\em J. Differential Geom. 32}, 2 (1990), 473--490.

\bibitem{MR2362821}
{\sc Losev, I.~V.}
\newblock Computation of {C}artan spaces for affine homogeneous spaces.
\newblock {\em Mat. Sb. 198}, 10 (2007), 31--56.

\bibitem{Marsden}
{\sc Marsden, J.~E., Ratiu, T., and Raugel, G.}
\newblock Symplectic connections and the linearisation of {H}amiltonian
  systems.
\newblock {\em Proc. Roy. Soc. Edinburgh Sect. A 117}, 3-4 (1991), 329--380.

\bibitem{MayrandTG}
{\sc Mayrand, M.}
\newblock Stratified hyperk\"ahler spaces from semisimple {L}ie algebras.
  ar{X}iv:1709.09126 (2017), 19 pp. {T}o appear in {T}ransformation {G}roups.

\bibitem{MR842398}
{\sc Mikityuk, I.~V.}
\newblock Integrability of invariant {H}amiltonian systems with homogeneous
  configuration spaces.
\newblock {\em Mat. Sb. (N.S.) 129(171)}, 4 (1986), 514--534, 591.

\bibitem{Nakajima1}
{\sc Nakajima, H.}
\newblock Instantons on {ALE} spaces, quiver varieties, and {K}ac-{M}oody
  algebras.
\newblock {\em Duke Math. J. 76}, 2 (1994), 365--416.

\bibitem{Nakajima2}
{\sc Nakajima, H.}
\newblock Quiver varieties and {K}ac-{M}oody algebras.
\newblock {\em Duke Math. J. 91}, 3 (1998), 515--560.

\bibitem{MR1708594}
{\sc Panyushev, D.~I.}
\newblock Complexity and rank of actions in invariant theory.
\newblock {\em J. Math. Sci. (New York) 95}, 1 (1999), 1925--1985.
\newblock Algebraic geometry, 8.

\bibitem{MR0294336}
{\sc Richardson, Jr., R.~W.}
\newblock Principal orbit types for algebraic transformation spaces in
  characteristic zero.
\newblock {\em Invent. Math. 16\/} (1972), 6--14.

\bibitem{Tauvel}
{\sc Tauvel, P., and Yu, R. W.~T.}
\newblock {\em Lie algebras and algebraic groups}.
\newblock Springer Monographs in Mathematics. Springer-Verlag, Berlin, 2005.

\bibitem{TimashevBook}
{\sc Timashev, D.~A.}
\newblock {\em Homogeneous spaces and equivariant embeddings}, vol.~138 of {\em
  Encyclopaedia of Mathematical Sciences}.
\newblock Springer, Heidelberg, 2011.
\newblock Invariant Theory and Algebraic Transformation Groups, 8.

\end{thebibliography}

\end{document}